\documentclass[11pt]{amsart}
\usepackage{tikz}
\usepackage{tkz-graph}
\usepackage{pgf}
\usepackage{amscd,amssymb,latexsym,verbatim}
\usepackage{hyperref}
\usepackage{graphicx}
\usepackage[T1]{fontenc}
\usepackage{multirow}
\usepackage{enumitem}
\usepackage{subfigure}
\usepackage[shortcuts]{extdash}

\usetikzlibrary{shapes,matrix,calc,arrows,decorations.markings,knots,patterns,intersections,cd}

\theoremstyle{plain}
\newtheorem{thm0}{Theorem}

\newtheorem*{thmmon}{Monodromy Theorem}
\newtheorem{thm}{Theorem}[section]
\newtheorem{lem}{Lemma}[section]

\newtheorem{prop}{Proposition}[section]
\newtheorem{cor}{Corollary}[section]

\theoremstyle{remark}
\newtheorem{rem}{Remark}[section]

\theoremstyle{definition}
\newtheorem{dfn}{Definition}[section]
\newtheorem{ntc}{Notation}[section]
\newtheorem{ntc0}[thm0]{Notation}

\newtheorem{example}[thm]{Example}

\makeatletter
\let\c@lem\c@thm
\let\c@prop\c@thm
\let\c@claim\c@thm
\let\c@cor\c@thm
\let\c@rem\c@thm
\let\c@dfn\c@thm
\let\c@prps\c@thm
\let\c@notation\c@thm
\makeatother

\def\makeautorefname#1#2{\expandafter\def\csname#1autorefname\endcsname{#2}}

\makeautorefname{lem}{Lemma}%
\makeautorefname{rem}{Remark}%
\makeautorefname{thm}{Theorem}%
\makeautorefname{thm0}{Theorem}%
\makeautorefname{dfn}{Definition}%
\makeautorefname{claim}{Claim}%
\makeautorefname{example}{Example}%
\makeautorefname{prop}{Proposition}%
\makeautorefname{cor}{Corollary}%
\makeautorefname{section}{Section}%
\makeautorefname{prps}{Properties}%
\makeautorefname{notation}{Notation}%

\newcommand{\cF}{\mathcal{F}}

\newcommand{\V}{V}
\newcommand{\E}{E}

\newcommand{\CC}{\mathbb{C}}

\newcommand{\KK}{\mathbb{K}}

\newcommand{\bz}{\mathbb{Z}}

\newcommand{\RR}{\mathbb{R}}

\newcommand{\bt}{\mathbb{T}}

\newcommand{\one}{\mathbf{1}}

\newcommand{\n}{r}

\DeclareMathOperator{\tor}{Tor}

\DeclareMathOperator{\coker}{Coker}
\DeclareMathOperator{\im}{Im}

\DeclareMathOperator{\rk}{Rank}

\DeclareMathOperator{\lcm}{lcm}

\title{Module structure of the homology of right-angled Artin kernels}

\author[E.~Artal]{Enrique Artal Bartolo}

\author[J.I.~Cogolludo]{Jos\'e Ignacio Cogolludo-Agust{\'\i}n}
\address{Departamento de Matem\'aticas, IUMA, Facultad de Ciencias\\
Universidad de Zaragoza\\
c/ Pedro Cerbuna 12\\
E-50009 Zaragoza SPAIN}
\email{artal@unizar.es,jicogo@unizar.es}

\author[S. L{\'o}pez de Medrano]{Santiago L{\'o}pez de Medrano}
\address{Instituto de Matemáticas\\ 
Universidad Nacional Autónoma de México\\
Ciudad Universitaria1\\
04510 México, CDMX\\
MEXICO  }
\email{santiago@im.unam.mx}

\author[D.~Matei]{Daniel Matei}
\address{Institute of Mathematics of the Romanian Academy\\
P.O. Box 1-764\\ RO-014700 Buch\-arest\\ Romania}
\email{Daniel.Matei@imar.ro}

\thanks{First two named authors are partially supported by
MTM2016-76868-C2-2-P and Grupo ``\'Algebra y Geometr{\'\i}a'' of
Gobierno de Arag\'on/Fondo Social Europeo. The third named author thanks the Universidad de Zaragoza for its hospitality during the preparationo of this work.
The fourth named author was partially supported by the Romanian Ministry of National Education,
CNCS-UEFISCDI, grant PNII-ID-PCE-2012-4-0156}

\begin{document}

\begin{abstract}
In this paper, we study the module structure of the homology of Artin kernels, i.e., kernels of 
non-resonant characters from right-angled Artin groups onto the integer numbers, the module
structure being with respect to the ring $\KK[t^{\pm 1}]$, where $\KK$ is a field
of characteristic zero. Papadima and Suciu
determined some part of this structure by means of the flag complex of the graph of the Artin group.
In this work, we provide more properties of the torsion part of this module, e.g., the dimension
of each primary part and the maximal size of Jordan forms (if we interpret the torsion structure
in terms of a linear map). These properties are stated in terms of homology properties of
suitable filtrations of the flag complex and suitable double covers of an associated toric complex.
\end{abstract}

\maketitle

\section*{Introduction}

The homological properties of Artin groups have been extensively studied in the recent years, see \cite{PS:09,SV:13} 
and references therein. These groups are especially attractive because of the combinatorial nature of their definition 
that might influence the properties of the groups. In this work we deal with a particular type of Artin groups called 
right-angled Artin groups, or RAAG for short, and the goal is to describe the homology of the kernels of most integer 
characters of these groups in combinatorial terms.

The definition of a RAAG can be given starting with a simplicial graph $\Gamma=(\V,\E)$, where $\V$ is the set of 
vertices and $\E\subset {\binom{\V}{2}}$ is the set of edges (seen as subsets of two elements in~$\V$).
The RAAG associated with $\Gamma$ can be defined via a finite presentation as
\begin{equation}\label{eq:ag}
G_\Gamma:=
\langle
v\in \V\mid [v,w]=1\text{ if }\{v,w\}\in \E
\rangle.
\end{equation}
Given a surjective character $\chi:G_\Gamma\twoheadrightarrow\mathbb{Z}$, one is interested in the \emph{Artin kernel}
$A_\Gamma^\chi:=\ker\chi$. As a first approach one can consider their homology groups with a module structure given as 
follows. Let $t\in G_\Gamma$ such that $\chi(t)=1$; conjugation by~$t$ defines an automorphism of~$A_\Gamma^\chi$ which induces 
a $\KK[t^{\pm 1}]$-module structure on $H_k(A_\Gamma;\KK)$ for any field~$\KK$; this structure does not depend
on the particular choice of~$t$. The main goal of this paper is to describe the structure of such modules. 

In particular, detailed properties on their structure will be given when $\KK$ has characteristic~$0$ 
and the character~$\chi$ satisfies $\chi(v)\neq 0$, $\forall v\in \V$,
a \emph{non-resonant} character. Moreover, a complete description will be given 
for $k\leq 2$ in terms of geometric properties of objects associated with $\Gamma$ and~$\chi$.

This problem was considered in~\cite{PS:09} and the following result summarizes the starting point of this paper.

\begin{thm0}[\cite{PS:09}]\label{thm:ps}
\[
H_{k+1}(A^\chi_\Gamma;\KK)\cong\KK[t^{\pm 1}]^{\n_k}
\bigoplus_{d\in\mathbb{Z}_{>0}}\bigoplus_{j>0}
\left(\KK[t^{\pm 1}]/\langle\Phi_d^j\rangle\right)^{\n_{k,j}(d)}
\] 
where $\Phi_d$ is the $d$-cyclotomic polynomial and the exponents $\n_k,\n_{k,j}(d)$ depend on~$\Gamma$.
\end{thm0}

Our purpose is to provide more properties concerning the torsion part. A general strategy is proposed 
to study the $\Phi_d$\=/primary part 
of this module $H_{k+1}(A^\chi_\Gamma;\KK)$ by reducing it to the study of the $(t+1)$-primary part of 
the homology of $H_{k+1}(A^\rho_\Gamma;\KK)$ with respect to an even character $\rho$ associated with~$\chi$.
The study of the $\Phi_2$-primary part for each $H_{k+1}$ is done in terms of a filtration of the flag complex
depending on $k$ and~$\rho$ and the two-fold cover defined by~$\rho$. Complete formulas are given for $k=0,1$, and in a large family of cases, the results
obtained determine the torsion of the homology.

Nevertheless, these formulas are quite complicated and it is hard to interpret them in terms of the properties of $\Gamma$. In the case the flag complex defined by $\Gamma$ satisfies certain acyclicity conditions more precise
formulas can be given for the exponents in \autoref{thm:ps}. These formulas completely determine the 
Artin exponents for $k=0,1$ as shown in \autoref{thm:main}. For $k>1$, we obtain the dimension of the $\Phi_d$-primary part, the number of its factors and
the exponent $r_{k,k+2}(d)$. Moreover, if this exponent vanishes we can detect the maximal $j$ such that $r_{k,j}(d)>0$.

Let us rephrase this strategy in another language. Let us denote by $T_{k+1}(A^\chi_\Gamma):=\tor H_{k+1}(A^\chi_\Gamma;\KK)$; it is a finitely dimensional $\KK$-vector space which coincides with the whole homology when
the flag complex of $\Gamma$ is $k$-connected. The action of $t$ induces an endomorphism
$t:T_{k+1}(A^\chi_\Gamma)\to T_{k+1}(A^\chi_\Gamma)$. The Jordan form of this endomorphism is equivalent 
to the decomposition of the torsion in~\autoref{thm:ps}. As a combination of some results in~\cite{PS:09}
and part of the results of this work, we obtain a version of the Monodromy Theorem for complex singularities.

\begin{thmmon}
Let $\chi$ be a non-resonant character. Let $t:T_{k+1}(A^\chi_\Gamma)\to T_{k+1}(A^\chi_\Gamma)$ be the endomorphism
defined by the monodromy action. Then 
\begin{enumerate}[label=\rm(\arabic*)]
\item\label{mon1} The characteristic polynomial of $t$ is a product of cyclotomic polynomials. More precisely,
the roots of the characteristic polynomial are roots of unity whose order is a divisor of $\chi(v)$ for some 
$v\in V$.
\item\label{mon2} The Jordan blocks for the eigenvalue $1$ have size~$1$.
\item\label{mon3} The Jordan blocks of $t$ acting on $T_{k+1}(A^\chi_\Gamma)$ have size at most~$k+2$.
\end{enumerate}
\end{thmmon}

The items \ref{mon1} and \ref{mon2} have been proved in~\cite{PS:09}; \ref{mon3} is a consequence of part this work,
see~\autoref{cor:Jordan_max}. In addition, using \autoref{prop:maxJordan} the number of such blocks can be recovered. 
Moreover, \autoref{prop:maxJordan1} explains how to recover the actual maximal size of Jordan blocks
in each case. Also, formulas in \autoref{thm:bkl} allow to recover the dimension of each primary part.
The dimension of the eigenspaces can be recovered from \autoref{thm:lk}. All these results are specially simple to 
state in the particular case when the flag complex is $k$-acyclic, see~\autoref{thm:main}.

The moment-angle variety $\bt$ of the flag complex of $\Gamma$ with respect to $(\mathbb{S}^1,1)$, see~\eqref{eq:ac},
is an Eilenberg-McLane space for $G_\Gamma$, which can be naturally realized as a subvariety of $(\mathbb{S}^1)^V$. 
Let us denote by $f^\chi$ the restriction of the monomial $\displaystyle\prod_{v\in V}t_v^{\chi(v)}$ to~$\bt$. The pull-back
\begin{equation}\label{eq:pull-back}
\begin{tikzcd}
\mathbb{T}^\chi\ar[r]\ar[d,"\tilde{\chi}" left]&\RR\ar[d]\\
\mathbb{T}\ar[r,"f^\chi"]&\mathbb{S}^1
\end{tikzcd}
\end{equation}
yields the infinite cyclic cover defined by $\chi$, whose fundamental group is $A^\chi_\Gamma$.
The moment-angle variety with respect to $(\CC^*,1)$ has the same homotopy type as~$\bt$; the map
$f^\chi$ can be defined in this setting and is a holomorphic map, which connects the classical 
Monodromy Theorem with the one above.

The paper is organized as follows. In \autoref{sec:settings} we give the classical description of the toric
complex $\bt$ associated with a RAAG, say $G_\Gamma$. This complex provides an Eilenberg-McLane space of $G_\Gamma$ 
and hence, its homological properties and those of cyclic covers will be key for our purposes. In particular, 
since $\KK[t^{\pm 1}]$ is a principal ideal domain, the space $H_{k+1}(A^\chi_\Gamma;\KK)$,
as a finitely generated $\KK[t^{\pm 1}]$-module, decomposes as a direct sum of its free and
torsion parts, which will be the focus of this work. The required tools to attack this problem will be presented in
\autoref{sec:torsion}. In this expository section we recover results by Papadima-Suciu~\cite{PS:09} on the 
$(t-1)$-torsion part of the homology $H_{k+1}(A^\chi_\Gamma;\KK)$ in terms of the graph~$\Gamma$. 
The Fitting ideals of the module $H_{k+1}(A^\chi_\Gamma;\KK)$ are considered in \autoref{sec:fitting}, the
purpose is to recover its $\Phi_d$-primary torsion part in terms of both the flag complex associated with 
$\Gamma$ and a weight filtration coming from $d$. This will give a precise bound of $k+2$ for the maximal 
order of $\Phi_d$ as a torsion element in $H_{k+1}(A^\chi_\Gamma;\KK)$ as well as a formula for the weighted 
sum of exponents $\sum_{j=1}^{k+2}\n_{k,j}(d)$ in terms of topological invariants of the weight filtration. 

In \autoref{sec:even} we introduce the concept of \emph{even characters} and prove a reduction 
of the general problem of studying the $\Phi_d$-primary part associated with a given character $\chi$
to the $(t+1)$-primary part associated with an even character $\rho_{\chi,d}$ that depends on $\chi$
and $d>1$. This allows us to give explicit formulas for the sum of the exponents 
$\sum_{j=1}^{k+2}\n_{k,j}(d)$ in terms of the $k$-homology of a double cover of the toric complex $\bt$
associated with~$\Gamma$. This reduction is key in the effective computation of the torsion invariants,
however explicit formulas are too intricate to be presented here. In the particular case when 
the flag complex $\cF$ has some acyclicity properties, explicit formulas for the torsion exponents
are presented in \autoref{sec:main}. In the final \autoref{sec:examples} we present some 
characteristic examples that describe the reduction strategy to study the torsion part in terms of
topological properties of the flag complex and its filtrations.

\section{Settings}
\label{sec:settings}
\numberwithin{equation}{section}

\subsection{The toric complex \texorpdfstring{$\bt$}{T}}
\label{sec:toric}
\mbox{}

Most of this section appears in~\cite{PS:09}; we recall the parts we will need.
As in the Introduction, we fix a finite simplicial graph $\Gamma$, whose set of vertices is denoted by $\V$ and 
set of edges $\E\subset {\binom{\V}{2}}$. We define the right-angled group $G_\Gamma$ as in~\eqref{eq:ag}.

Let $\cF$ be the flag complex of~$\Gamma$. This is a simplicial complex with set of vertices~$\V$ and such that 
$X=\{v_0,v_1,\dots,v_r\}$ is an $r$-simplex of~$\cF$ if and only if $X$ is a clique of $\Gamma$.
For the sake of simplicity we will work over a field $\KK$ of characteristic 0.
We will denote by $(\tilde{C}_*(\cF),\partial)$ the augmented chain complex of $\cF$ over $\KK$, 
i.e., $\tilde{C}_{-1}(\cF)=\KK$ and $\partial:\tilde{C}_{0}(\cF)\to\tilde{C}_{-1}(\cF)$ sends any vertex to~$1$.

Let us associate more objects to~$\Gamma$. The moment-angle complex $\bt$ of~$\cF$ associated to the pair $(\mathbb{S}^1,1)$
is the $CW$-subcomplex of $(\mathbb{S}^1)^{\V}$ defined as 
\begin{equation}\label{eq:ac}
\bt:=\bigcup_{\sigma\in\cF} (\mathbb{S}^1)^{\sigma}
\end{equation}
We consider the empty simplex as a $(-1)$-dimensional simplex and by convention
$(\mathbb{S}^1)^{\emptyset}=\one:=(1)_{v\in\V}$. The circle~$\mathbb{S}^1$
has a natural $CW$-cell decomposition with cells $\{1\}$ and $\mathbb{S}^1\setminus\{1\}$.
It induces a $CW$-complex structure on~$\bt$ such that an $r$-simplex $\sigma\in\cF$ determines
an $(r+1)$-cell 
\[
\sigma^\bt:=\left(\mathbb{S}^1\setminus\{1\}\right)^\sigma\times\{1\}^{\V\setminus\sigma}.
\]
In particular, $\emptyset^\bt=\{\one\}$ is the only~$0$-cell. For each~$v\in\V$, $\overline{v^\bt}$
is a circle and for each edge $e=\{v,w\}$, $\overline{e^\bt}$ is a two-dimensional torus where 
$\overline{v^\bt}$, $\overline{w^\bt}$ are two circles intersected at one point.
We have the following classical result.

\begin{prop}[\cite{ChD:95,MevW:95}]
The fundamental group of $\bt$ is isomorphic to $G_\Gamma$. Moreover,  $\bt$
is an Eilenberg-McLane space for this group.
\end{prop}

%

The cellular complex $(C_*(\bt),\partial_\bt)$
is minimal (i.e., $\partial_\bt=0$) and, if we drop the differential maps,
$C_*(\bt)\cong\widetilde{C}_{*-1}(\cF)$. 
In particular,
$H_*(\bt;\KK)=\tilde{C}_{*-1}(\cF)$ as the differential of $C_*(\bt)$ vanishes. 

\subsection{The associated \texorpdfstring{$\bt^\chi$}{T} complex}
\mbox{}

In the Introduction we have considered a surjective character
$\chi:G_\Gamma\to\bz$; for $v\in\V$ we denote $n_v:=\chi(g_v)$.
Note that this character is completely determined by the 
tuple $(n_v)_{v\in\V}$.
For technical reasons we impose the restrictions $n_v\neq 0$, $\forall v\in\V$.
Without loss of generality, we are going to assume that
$n_v>0$. The epimorphism
condition is equivalent to $\gcd_{v\in\V}\{n_v\}=1$. 

Let $\tilde{\chi}:\bt^\chi\to\bt$ be the
infinite cyclic covering associated to~$\chi$ as in~\eqref{eq:pull-back}. The fiber $\tilde{\chi}^{-1}(\one)$ is identified
with~$\bz$. We pick up a generator $t:\bt^{\chi}\to\bt^{\chi}$ of the deck automorphism
group of $\tilde{\chi}$, such that if $\tilde{\one}\in\tilde{\chi}^{-1}(\one)$
is associated to~$0$, then the identification $\bz\equiv\tilde{\chi}^{-1}(\one)$ is given
by $n\equiv t^n(\tilde{\one})$. Recall that $\pi_1(\bt^\chi)\cong\ker\chi=A_\Gamma^\chi$, the Artin kernel. 

The $CW$-complex structure
of $\bt$ induces another such structure in $\bt^\chi$ and  $(C_*(\bt^\chi),\partial_\chi)$ is 
a complex of $\KK[t^{\pm 1}]$-modules (the structure is induced by the deck automorphisms of the covering).
Note that $\bt^\chi$ is an Eilenberg-McLane space for $A_\Gamma$ and then 
$H_*(\bt^\chi;\KK)$ is isomorphic (as $\KK[t^{\pm 1}]$-module) with $H_*(A_\Gamma^\chi;\KK)$.

Let $\sigma\in\cF$ be a simplex; recall that $\sigma^\bt$ is a cell of $\bt$. Then,
the decomposition of
$\tilde{\chi}^{-1}(\sigma)$ in connected components is a disjoint union of cells.
We fix one of them and denote it by $\sigma^\chi$; below we will describe the concrete
choice for each~$\sigma^\chi$. Then
\[
\tilde{\chi}^{-1}(\sigma)=\bigcup_{n\in\bz} t^n\cdot \sigma^{\chi}.
\]
With its $\KK[t^{\pm 1}]$-module structure, $C_*(\bt^\chi)$ is free and
there is a basis $\{\sigma^\chi\mid\sigma\in\cF\}$. Note that 
$\dim\sigma^\chi=\dim\sigma+1$.

By fixing an order in the set of vertices $V$ one can define an \emph{incidence number} 
$\langle\sigma|\tau\rangle$ for any $\sigma,\tau\in\cF$ as follows
$$\langle\sigma|\tau\rangle = 
\begin{cases}
(-1)^s & \text{ if } \tau = \sigma\setminus \{v_i\}, \text{ and } s=\# \{j \mid v_j\in \tau, i<j\}\\
0 & \text{ otherwise.} 
\end{cases}$$
This number fits in the boundary map formula as 
\[
\partial(\sigma)=\sum_{\tau\in\cF}\langle\sigma|\tau\rangle\tau=
\sum_{v\in\sigma}\langle\sigma|\sigma_{v}\rangle\sigma_{v}.
\]
where, for any $v\in \sigma$, $\sigma_{v}:=\sigma\setminus\{v\}$. 
Then, once $\emptyset^\chi$ has been fixed,
there is a unique choice of oriented cells $\sigma^\chi$, such that
\begin{enumerate}[label=($\chi$\arabic{enumi})]
\item\label{chi1} $\partial^\chi(v^\chi)=(t^{n_v}-1)\cdot \emptyset^\chi$;
\item\label{chi2} $\displaystyle\partial^\chi(\sigma^\chi)=\sum_{v\in\sigma}\left\langle\sigma|\sigma_{v}\right\rangle
(t^{v}-1)\sigma_{v}^\chi$.
\end{enumerate}
\autoref{fig:square} visualizes this choice in the case where
$\sigma=\{v,w\}$ is an edge in $\cF$.
Summarizing, a simplex $\sigma$ in $\cF$ determines elements 
$\sigma\in C_*(\cF)$, $\sigma^\bt\in C_*(\bt)$ and $\sigma^\chi\in C_*(\bt^\chi)$.

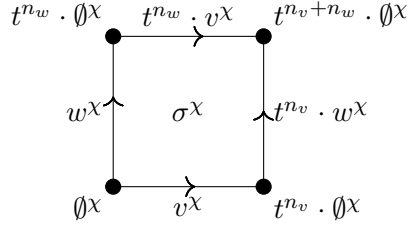
\begin{figure}[ht]
\begin{center}
\begin{tikzpicture}[scale=2,vertice/.style={draw,circle,fill,minimum size=0.2cm,inner sep=0}]
\tikzset{flecha1/.style={decoration={
  markings,
  mark=at position #1 with  {\arrow[scale=1.5]{>}}},postaction={decorate}}}
\tikzset{flecha/.style={decoration={
  markings,
  mark=at position #1 with  {\arrow[scale=1.5]{<}}},postaction={decorate}}}

\coordinate (A1) at (0,0);
\coordinate (A2) at (1,0);
\coordinate (A3) at (1,1);
\coordinate (A4) at (0,1);
\draw[flecha1=0.15,flecha1=.4,flecha=.64,flecha=.88] (A1) rectangle (A3);
\foreach \x in {1,...,4}
{
\node[vertice] at (A\x) {};
}
\node[below left] at (A1) {$\emptyset^\chi$};
\node[below right] at (A2) {$t^{n_v}\cdot \emptyset^\chi$};
\node[above left] at (A4) {$t^{n_w}\cdot \emptyset^\chi$};
\node[above right] at (A3) {$t^{n_v+n_w}\cdot \emptyset^\chi$};
\node[below] at ($.5*(A1)+.5*(A2)$) {$v^\chi$};
\node[right] at ($.5*(A3)+.5*(A2)$) {$t^{n_v}\cdot w^\chi$};
\node[above] at ($.5*(A3)+.5*(A4)$) {$t^{n_w}\cdot v^\chi$};
\node[left] at ($.5*(A4)+.5*(A1)$) {$w^\chi$};
\node[] at ($.5*(A1)+.5*(A3)$) {$\sigma^\chi$};

\end{tikzpicture}
\caption{Special lift of the edge~$\sigma$}
\label{fig:square}
\end{center}
\end{figure}

\begin{ntc}
For $v\in\V$, we define $p_v:=t^{n_v}-1$, and for $\sigma$ we define 
$p_\sigma:=\prod_{v\in\sigma} p_v$. With this notation the formula in \ref{chi2}
is equivalent to:
\[
\frac{1}{p_\sigma}\partial^\chi(\sigma^\chi)=
\sum_{v\in\sigma}\left\langle\sigma|\sigma_{v}\right\rangle\frac{1}{p_{\sigma_v}}\sigma_v
\]
Given $k\in\bz$, for the morphism $\partial _k:\tilde{C}_k(\cF)\to\tilde{C}_{k-1}(\cF)$, its matrix in
the bases $\{\sigma\}_{k\text{-simplices}}$ and $\{\tau\}_{(k-1)\text{-simplices}}$ is the incidence
matrix $(\langle\sigma|\tau\rangle)_{\sigma,\tau}$.
\end{ntc}

\begin{rem}
The matrix of $\partial_{k+1}^\chi:C_{k+1}(\bt^\chi)\to C_{k}(\bt^\chi)$, as free $\KK[t^{\pm 1}]$-modules with bases
$\{\sigma^\chi\}_{k\text{-simplices}}$ and $\{\tau^\chi\}_{(k-1)\text{-simplices}}$
is $\left(\langle\sigma|\tau\rangle\frac{p_\sigma}{p_\tau}\right)_{\sigma,\tau}$. Note that the entries
are actually in $\KK[t^{\pm 1}]$, since $\langle \sigma|\tau\rangle=0$ when $\tau$ is not a face of $\sigma$.
If $\tau=\sigma_v$ the entry is $\langle\sigma|\sigma_v\rangle p_v$.
\end{rem}

\subsection{First results on the module structure of Artin kernels}
\mbox{}

Most of the results in this section come from~\cite{PS:09}; they are included since they help the reading.
Since $H_*(\bt^\chi;\KK)$ is a finitely generated $\KK[t^{\pm1}]$-module, eventually 
replacing $\KK$ by a finite extension, there is an isomorphism
\[
H_{k+1}(\bt^\chi;\KK)\cong\KK[t^{\pm 1}]^{\n_k}\oplus\bigoplus_{\lambda\in\KK^*} H_{k,\lambda},
\]
where,
\begin{equation}
\label{eq:Hkj}
H_{k,\lambda}=\bigoplus_{j=1}^\infty \left(\KK[t^{\pm 1}]\Big/\langle(t-\lambda)^j\rangle\right)^{\n_{k,j}^\lambda}.
\end{equation}
In order to recover each term of the direct sum it is useful to compare this homology with the homology
of some tensored complexes. Recall that $H_{k+1}(\bt^\chi;\KK)=H_{k+1}(C_*(\bt^\chi))$.

\begin{lem}\label{lem:uct}
Let $R\supset\KK[t^{\pm 1}]$ be an integral domain with its natural structure as $\KK[t^{\pm 1}]$-module.
Then, there is a natural isomorphism
\[
H_{k+1}(C_*(\bt^\chi)\otimes_{\KK[t^{\pm 1}]} R)\cong 
H_{k+1}(C_*(\bt^\chi))\otimes_{\KK[t^{\pm 1}]} R.
\]
\end{lem}

\begin{proof}
The Universal Coefficient Theorem yields the following exact sequence:
\[
0\to H_{k+1}(C_*(\bt^\chi))\otimes_{\KK[t^{\pm 1}]} R\to
H_{k+1}(C_*(\bt^\chi)\otimes_{\KK[t^{\pm 1}]} R)\to 
\tor(H_{k}(C_*(\bt^\chi)),R)\to 0.
\]
If $A,B$ are $\KK[t^{\pm 1}]$-modules then $\tor(A\oplus B,R)=\tor(A,R)\oplus\tor(B,R)$.
Since $R$ is an integral domain, from its very definition $\tor(\KK[t^{\pm 1}]/\langle p(t)\rangle,R)=0$,
and hence the last term of the sequence vanishes.
\end{proof}

\begin{lem}\label{lem:invertible}
Let $R$ be as in Lemma{\rm~\ref{lem:uct}}. Assume that $\forall v\in V$, the polynomial $p_v(t)=t^{n_v}-1$
is invertible in $R$. Then,
\[
R^{r_k}\cong H_{k+1}(C_*(\bt^\chi)\otimes_{\KK[t^{\pm 1}]} R)\cong\tilde{H}_k(\cF;R).
\]
\end{lem}

\begin{proof}
From the hypotheses, $\left\{\sigma^\chi\otimes\frac{1}{p_\sigma}\right\}_{k\text{-simplices}}$ and 
$\left\{\tau^\chi\otimes\frac{1}{p_\tau}\right\}_{(k-1)\text{-simplices}}$
are bases of $C_{k+1}(\bt^\chi)\otimes_{\KK[t^{\pm 1}]} R$ and
$C_{k}(\bt^\chi)\otimes_{\KK[t^{\pm 1}]} R$, respectively.
The matrix of $\partial_{k+1}^{\chi}\otimes 1_{R}$ on those
bases is 
$\left(\langle\sigma|\tau\rangle\right)_{\sigma,\tau}$. As a consequence,
this complex is naturally isomorphic to 
$
(C_{*-1}(\cF)\otimes_{\KK}  R,\partial)
$
which yields the second automorphism. The first one is a consequence of Lemma~\ref{lem:uct} 
and the hypotheses of the statement and the fact that $\tilde{H}_k(\cF;R)$ is free.
\end{proof}

We recover two results in~\cite{PS:09}.

\begin{enumerate}[label=(PS\arabic{enumi})]
\item \label{PS1} The number $\n_k$ coincides with $\dim_\KK \tilde{H}_k(\cF;\KK)$. 
\item \label{PS2} 
If $\lambda$ is not a root of unity of order a divisor of some $n_v$, $v\in V$, then $\n_{k,j}^\lambda=0$.
\end{enumerate}

The proof of~\ref{PS1}
is a consequence of Lemma~\ref{lem:invertible}, for $R=\KK(t)$.
The proof of~\ref{PS2} is a consequence of Lemma~\ref{lem:invertible} for 
the localized ring
$R=\KK[t^{\pm 1}]_{\langle t-\lambda\rangle}$. Using Lemma~\ref{lem:uct}, we deduce that there is no torsion
for these $\lambda$.
A third result from~\cite{PS:09} is immediately available

\begin{enumerate}[label=(PS\arabic{enumi})]
\setcounter{enumi}{2}
\item \label{PS3} If $\cF$ is $n$-connected, then 
$\dim_\KK H_{m}(\bt^\chi;\KK)<\infty$ for $0\leq m\leq n+1$. Moreover, if $\tilde H_{n+1}(\cF;\KK)\neq 0$,
then $H_{n+2}(\bt^\chi;\KK)$ has infinite dimension.
\end{enumerate}

\section{The torsion part of the homology}
\label{sec:torsion}

\subsection{Primary parts via localization}
\mbox{}

Fix $\lambda\in\KK^*$ a primitive $d$-root of unity for some $d$ dividing $\lcm\{n_v\}_{v\in\V}$. 

There is a natural monomorphism 
$\eta_\lambda:\KK[t^{\pm 1}]\to\KK[[s]]$ such that $t\mapsto\lambda+s$. 
Note that 
\[
\eta_\lambda(t^n-1)=s^{\alpha_\lambda} q_n(s)\text{ where } q_n\in\KK[[s]]^\times\text{ and }
\alpha_\lambda=
\begin{cases}
0&\text{ if }\lambda^n\neq 1\\
1&\text{ if }\lambda^n=1.
\end{cases}
\]
The use of this map has been borrowed from~\cite{PS:09}.
As a consequence
\[
\eta_\lambda(p_\sigma)=s^{\omega(\lambda,\sigma)} u_\sigma(s) \text{ where } u_\sigma\in\KK[[s]]^\times\text{ and }
\omega(\lambda,\sigma)=\#\{v\in\sigma\mid \lambda^{n_v}=1\}.
\]
Note that $\omega(\lambda,\sigma)$ depends on $d$ rather than a particular choice of the primitive $d$-root of unity.
Hence we denote $\omega_d(\sigma):=\omega(\lambda,\sigma)$.

The following lemma is useful for the study of the torsion part of the homology of $\bt^\chi$.

\begin{lem}
\label{lem:torsion}
The torsion part of  $H_{k+1}(C_*(\bt^\chi)\otimes_{\eta_\lambda}\KK[[s]])$
coincides with the torsion part of $\coker(\partial_{k+2}^\chi\otimes 1_{\KK[[s]]})$
and with
\[
H_{k,\lambda}\otimes_{\eta_\lambda}\KK[[s]]=
\bigoplus_{j=1}^\infty \left(\KK[[s]]\Big/\langle s^j\rangle\right)^{\n_{k,j}^\lambda}.
\]
\end{lem}

\begin{proof}
The following exact sequence holds:
\[
\begin{tikzcd}
0\ar[r]&[-16pt] H_{k+1}(C_*(\bt^\chi)\otimes_{\eta_\lambda}\KK[[s]])\ar[r]&[-11pt]
\coker\partial^\chi_{k+2}\otimes 1_{\KK[[s]]}\ar[r]&[-11pt]
\dfrac{C_{k+1}(\bt^\chi)\otimes_{\eta_\lambda}\KK[[s]]}{\ker\partial_{k+1}^\chi\otimes 1_{\KK[[s]]}}\ar[r]
&[-18pt] 0\\[-30pt]
&\rotatebox{90}{$=$}&\rotatebox{90}{$=$}&\rotatebox{-90}{$\cong$}&\\[-30pt]
&\dfrac{\ker\partial_{k+1}^\chi\otimes 1_{\KK[[s]]}}{\im\partial^\chi_{k+2}\otimes 1_{\KK[[s]]}}
&\dfrac{C_{k+1}(\bt^\chi)\otimes_{\eta_\lambda}\KK[[s]]}{\im\partial^\chi_{k+2}\otimes 1_{\KK[[s]]}}
&\im\partial^\chi_{k+1}\otimes 1_{\KK[[s]]}&
\end{tikzcd}
\]
Since the last module is free, the torsion part of $H_{k+1}(C_*(\bt^\chi)\otimes_{\eta_\lambda}\KK[[s]])$ 
coincides with the one of 
$\coker\partial^\chi_{k+2}\otimes 1_{\KK[[s]]}$.
The matrix of $\partial^\chi_{k+2}\otimes 1_{\KK[[s]]}$ is a presentation matrix for $\coker\partial^\chi_{k+2}\otimes 1_{\KK[[s]]}$.
The last statement is a consequence of Lemma~\ref{lem:uct}.
\end{proof}

\begin{lem}
The matrix of the differential $\partial_{k+2}^\chi\otimes 1_{\KK[[s]]}$
is $\left(\langle\sigma|\tau\rangle s^{\omega_d(\sigma)-\omega_d(\tau)}\right)_{\sigma,\tau}$,
for the bases $\left\{\frac{\sigma^\chi\otimes 1}{u_\sigma}\right\}_{k\text{-simplices}}$ and $\left\{\frac{\tau^\chi\otimes 1}{u_\tau}\right\}_{(k-1)\text{-simplices}}$.
\end{lem}

\begin{proof}
It is not hard to check that the matrix in
the bases 
$\{\sigma^\chi\otimes 1\}_{k\text{-simplices}}$ and $\{\tau^\chi\otimes 1\}_{(k-1)\text{-simplices}}$ is equal to
$\left(\langle\sigma|\tau\rangle s^{\omega_d(\sigma)-\omega_d(\tau)}\frac{u_\sigma}{u_\tau}\right)_{\sigma,\tau}$. The result follows.
\end{proof}

It is clear that the values $\omega_d(\sigma)$ are interesting. The following result is straightforward.

\begin{lem}$\omega_d(\sigma)=\#\{v\in\sigma\mid d \text{ divides } n_v\}$.
\end{lem}

\subsection{The \texorpdfstring{$(t-1)$}{(t-1)}-primary part}
\mbox{}

This corresponds with considering the case $d=1$ in the construction above. 
We will recover another result in~\cite{PS:09}.
\begin{enumerate}[label=(PS\arabic{enumi})]
\setcounter{enumi}{3}
\item \label{PS4} $H_{k,1}=\left(\KK[t^{\pm 1}]\Big/\langle t-1\rangle\right)^{\rk\partial_{k+1}}$. 
\end{enumerate}

We need to compute a matrix for the differential $\partial_{k+2}^\chi\otimes 1_{\KK[[s]]}$
for the morphism $\eta_1:\KK[t^{\pm1}]\to\KK[[s]]$. Such a matrix is
$s\left(\langle\sigma|\tau\rangle \right)_{\sigma,\tau}$ and~\ref{PS4} follows. Note that, in particular,
the $(t-1)$-primary part is semisimple.

\subsection{The \texorpdfstring{$(t-\lambda)$}{(t-lambda)}-primary part}
\mbox{}

Consider $\lambda$ be a $d$-primitive root of unity for $d>1$. The $(t-\lambda)$-primary part of 
$H_{k+1}(A_\Gamma^\chi;\KK)$ can be recovered using the following result.

\begin{prop}
Let $\lambda$ be a $d$-primitive root of unity. 
The exponents $\n_{k,j}^\lambda$ depend completely on the function $\omega_d$ (in fact, on ${\omega_d}_{|\V}$). 
These exponents $\n_{k,j}^\lambda$ coincide with the exponents $\n_{k,j}(d)$ of {\rm~\autoref{thm:ps}}
\end{prop}

We will call this part in \autoref{thm:ps} the $\Phi_d$-torsion part of $H_{k+1}(A_\Gamma^\chi;\KK)$.
By \autoref{lem:torsion} this submodule is completely determined by the torsion part of the module 
$\coker(\partial_{k+2}^\chi\otimes 1_{\KK[[s]]})$. A matrix of this map is 
$\left(\langle \sigma,\tau\rangle s^{\omega_d(\sigma)-\omega_d(\tau)}\right)_{\sigma,\tau}$. 
This torsion part is determined by its Smith form since $\KK[[s]]$ is a Euclidean ring.

\begin{ntc0}
Let us assume that this Smith form has its first $r$ diagonal entries equal to~$1$
and then the other non-zero entries are $(s^{a_{k,1}},\dots,s^{a_{k,\ell_k}})$ where 
$0< a_{k,1}\leq\dots\leq a_{k,\ell_k}$, $\ell_k\geq 0$. In this case
\[
H_{k,\lambda}\cong\bigoplus_{j=1}^{\ell_k} \KK[t^{\pm 1}]\Big/\langle(t-\lambda)^{a_{k,j}}\rangle
\]
and $\n_{k,j}(d)$ is the number of entries such $a_{k,i}=j$.
In fact, if $\Phi_d$ is the cyclotomic polynomial whose roots are the $d$-primitive roots
of unity, then
\[
\bigoplus_{j=1}^{\ell_k} \KK[t^{\pm 1}]\Big/\langle\Phi_d(t)^{a_{k,j}}\rangle
\]
is a factor of $H_{k+1}(\bt^\chi;\KK)$.
\end{ntc0}

The goal will be to compute some of these numbers relating them to the 
sequence of non-zero Fitting ideals of the above matrix. The first $r$ Fitting
ideals equal $\KK[[s]]$, while the next ones will be of the form
$\langle s^{b_{k,1}}\rangle,\dots,\langle s^{b_{k,\ell_k}}\rangle$ where
\begin{equation}
\label{eq:bs}
b_{k,1}\leq\dots\leq b_{k,\ell_k},\qquad 
b_{k,j}=a_{k,1}+\dots+a_{k,j}.
\end{equation}
Our goal is to compute these numbers. Note that this means that we have
$r+\ell_k$ non-zero ideals.

\section{Fitting ideals}
\label{sec:fitting}
As in the previous section, we will fix $\lambda\in\KK$ a primitive $d$-th root of unity for $d>1$.
Prior to the computation of the minors of the matrix 
$\left(\langle \sigma,\tau\rangle s^{\omega_d(\sigma)-\omega_d(\tau)}\right)_{\sigma,\tau}$,
let us study its square submatrices in order to characterize non-zero minors.
For simplicity from now on the subindex in $\omega_d$ will be dropped. 
These submatrices are parametrized by two subsets of simplices:
the subset $K$ of $(k+1)$-simplices, corresponding to the rows of the submatrix,
and the subset $L$ of $k$-simplices, corresponding to the columns
\emph{not} in the submatrix. We denote these submatrices as $M_{K,L}^\chi$.
We will compare this submatrix with $M_{K,L}=\left(\langle\sigma|\tau\rangle\right)_{\sigma\in K}^{\tau\notin L}$. 
Since we deal with square matrices, $K\cup L$ is bijective with the set of $k$-simplices of~$\cF$.

Let us introduce some helpful notation:
\begin{itemize}
\item The $k$-skeleton of $\cF$ is denoted by $\cF^k$.
\item The set of simplices of dimension~$k$ of~$\cF$ is denoted as $\tilde{\cF}^k$.
\item For a subset $K$ of $\tilde \cF^k$ we define 
its \emph{weight} is defined as $\omega(K):=\sum_{\sigma\in K}\omega(\sigma)$.
\end{itemize}

\begin{lem}\label{lem:mult_KL}
$\det M_{K,L}^\chi=\det M_{K,L} s^{\omega(K)+\omega(L)-\omega(\tilde{\cF}^{k})}$.
\end{lem}

\begin{proof}
If we multiply each column of the matrix of $\partial_{k+2}^\chi\otimes 1_{\KK[[s]]}$ 
by $s^{\omega(\sigma)}$ and divide each row by $\omega(\tau)$, we obtain the matrix
for $\partial_{k+1}$. Hence,
\begin{equation*}
\det M_{K,L}=
\frac{s^{\omega(\tilde{\cF}^k})}{s^{\omega(L)}s^{\omega(K)}}
\det M_{K,L}^\chi
\qedhere
\end{equation*}
\end{proof}

%

The next goal is to characterize geometrically when these minors
do not vanish. Let us denote:
\[
M_K=\cF^k\cup\bigcup_{\sigma\in K}\sigma,\qquad 
M_L=\cF^{k-1}\cup\bigcup_{\tau\in L}\tau.
\]
Note that $\chi(M_K)-\chi(M_L)=(-1)^{k+1}\# K+(-1)^k(\#\tilde{\cF}^k-\# L)=0$,
i.e $\chi(M_K,M_L)=0$.

\begin{prop}
The minor  $\det M_{K,L}$ (and hence $\det M_{K,L}^\chi$) does not vanish if and only if
the pair $(M_K,N_L)$ is $\KK$-acyclic. 
\end{prop}

\begin{proof}
Let us denote $\KK\langle K\rangle$, $\KK\langle L\rangle$, the subspaces
of $C_{k+1}(\cF)$ and $C_{k}(\cF)$ generated by $K,L$, respectively. 
The relative complex is as follows:
\[
C_j(M_K,M_L)=
\begin{cases}
0&\text{ if }j<k\text{ or }j>k+1,\\
\KK\langle K\rangle&\text{ if }j=k+1,\\
\tilde{C}_k(\cF)\Big/\KK\langle L\rangle&\text{ if }j=k.
\end{cases}
\]
In the natural bases the matrix of the only relevant differential
is $M_{K,L}$, and the statement follows. 
\end{proof}

Let us express this acyclicity condition using the
long exact sequence of pairs:
\[
\begin{tikzcd}[column sep=9pt]
0&
\tilde{C}_{k+1}(\mathcal{F})&
0&
\tilde{C}_{k}(\mathcal{F})&
{}\\[-25pt]
\rotatebox{90}{$=$}&
\rotatebox{90}{\hspace{0mm}$\subset$}&
\rotatebox{90}{$=$}&
\rotatebox{90}{\hspace{0mm}$\subset$}&&
{}
\\[-25pt]
\tilde{H}_{k+1}(M_L)\ar[r]&
\tilde{H}_{k+1}(M_K)\ar[r]&
\tilde{H}_{k+1}(M_K,M_L)\ar[r]&
\tilde{H}_{k}(M_L)\ar[r]&
\tilde{H}_{k}(M_K)\ar[r]&[-10pt]
{}
\\[-5pt]
{}\ar[r]&
\tilde{H}_{k}(M_K,M_L)\ar[r]&
\tilde{H}_{k-1}(M_L)\ar[r]&
\tilde{H}_{k-1}(M_K)\ar[r]&
\tilde{H}_{k-1}(M_K,M_L)
\\[-25pt]
&\rotatebox{90}{$=$}&&
\rotatebox{90}{$=$}&
\rotatebox{90}{$=$}&
{}\\[-25pt]
&0&&
\tilde{H}_{k-1}(\mathcal{F})&
0
\end{tikzcd}
\]

\begin{prop}
The minor $M_{K,L}$ does not vanish if and only if
\begin{enumerate}[label=\rm(Ac\arabic{enumi})]
\item\label{Ac1} $\tilde{H}_{k+1}(M_K)=0$, i.e., $\ker\partial_{k+1}\cap\KK\langle K\rangle=0$.
\item\label{Ac2} The map $\tilde{H}_k(M_L)\to \tilde{H}_k(M_K)$, induced by the inclusion, is an isomorphism.
\end{enumerate}
\end{prop}

\begin{proof}
From the previous diagram $(\Rightarrow)$ is evident. For $(\Leftarrow)$, 
we need to prove that $\tilde{H}_{k+1}(M_K,M_L)=0=\tilde{H}_{k}(M_K,M_L)$ since the other equalities
are obvious. The two conditions of the statement 
imply $\tilde{H}_{k+1}(M_K,M_L)=0$. The Euler characteristic of the pair implies
$\tilde{H}_{k}(M_K,M_L)=0$.
\end{proof}

\begin{dfn}
We say that a pair $(K,L)$ is \emph{acyclic} if $(M_K,M_L)$ is acyclic.
An acyclic pair $(K,L)$ is \emph{minimal} if the sum of its multiplicities
is minimal among the acyclic pairs of the same size. The \emph{size} of an acyclic pair $(K,L)$
is $\#K$.
\end{dfn}

The next step in our strategy is to determine the maximal size of the non-zero minors.

\begin{prop}
\label{prop:acyclic}
The maximal size of an acyclic pair is $\rk\partial_{k+1}$. Any $(K,L)$ acyclic pair
of size $\rk\partial_{k+1}$ is obtained by the choice of $K,L$ satisfying
the following conditions:
\begin{enumerate}[label=\rm(\arabic{enumi})]
\item $\#K=\rk\partial_{k+1}$,
\item $\tilde{H}_k(M_K)=\tilde{H}_k(\cF)$,
\item $\#L=\#\tilde{\cF}^k-\rk\partial_{k+1}$,
\item The map $\tilde{H}_k(M_L)\to \tilde{H}_k(\cF)$, induced by inclusion, is an isomorphism.
\end{enumerate}
In particular, the conditions for $K$ and $L$ are independent.
\end{prop}

\begin{proof}
Let $(K,L)$ be an acyclic pair. Note that the condition~\ref{Ac1} implies that
\[
\#K\leq\#\tilde{\cF}^{k+1}-\dim\ker{\partial_{k+1}}=\dim\tilde{C}_{k+1}(\cF)-\dim\ker{\partial_{k+1}}=\rk\partial_{k+1}.
\]
It is clear that there exists $K\subset\tilde{\cF}^{k+1}$ such that $\#K=\rk\partial_{k+1}$
and $\tilde{C}_{k+1}(\cF)=\ker{\partial_{k+1}}\oplus\KK\langle K\rangle$.
Note that $\tilde{H}_{k}(M_K)=\tilde{H}_k(\cF)$. 

In order to satisfy~\ref{Ac2} we need to find~$L$ such that
$\tilde{H}_{k}(M_L)\to \tilde{H}_{k}(\cF)$ is an isomorphism. As the map
\[
\ker\partial_{k}=\tilde{H}_{k}(\cF^{k})\longrightarrow \tilde{H}_{k}(\cF)=\ker\partial_{k}/\im\partial_{k+1}
\]
is clearly surjective, any $L\subset\tilde{\cF}^k$ such that $\#L=\#\tilde{\cF}^k-\rk\partial_{k+1}$
and $\tilde{C}_{k}(\cF)=\im{\partial_{k+1}}\oplus\KK\langle L\rangle$.
\end{proof}

One can define a natural filtration given by
\begin{equation}
\label{eq:Fdef} 
\tilde{\cF}^m_{j}:=\{\sigma\in \cF^m\mid\omega(\sigma)\leq j\},\qquad
\cF^m_j:=\cF^{m-1}\cup\tilde{\cF}^m_{j},\quad 
0\leq j\leq m+1.
\end{equation}
Note that
\begin{equation}
\label{eq:Ffiltration} 
\array{l}
\cF^{k-1}\subset\cF^{k}_0\subset\cF^{k}_1\dots\subset\cF^{k}_{k}\subset\cF^{k}_{k+1}=
\\
\cF^{k}\subset\cF^{k+1}_{0}
\subset\cF^{k+1}_1\dots\subset\cF^{k+1}_{k+1}\subset\cF^{k+1}_{k+2}=\cF^{k+1}.
\endarray
\end{equation}
In what follows we will describe certain properties of minimal acyclic pairs in terms of the 
homology of blocks of this filtration.
\begin{lem}
\label{prop:kj}
Let $0\leq j\leq k+2$. There exists $K_j\subset\tilde{\cF}^{k+1}_j$
such that 
$\tilde{H}_{k+1}(M_{K_j})=0$ and $\tilde{H}_k(M_{K_j})\to \tilde{H}_k(\cF^{k+1}_j)$ is an isomorphism.
The cardinality of any such $K_j$ depends only on~$j$.
Its cardinality is 
\[
u_{k+1,j}:=\dim\tilde{C}_{k+1}({\cF}^{k+1}_{j})-\dim \tilde{H}_{k+1}(\cF^{k+1}_j)
\]
and
\begin{equation}\label{eq:kj}
u_{k+1,j}-u_{k+1,j-1}=\dim\tilde{H}_k(\cF^{k+1}_{j-1})-\dim\tilde{H}_k(\cF^{k+1}_{j}).
\end{equation}
\end{lem}

\begin{proof}
Let  $K_j\subset\tilde{\cF}^{k+1}_{j}$. The condition $\tilde{H}_{k+1}(M_{K_j})=0$ is equivalent to 
$\KK\langle K_j\rangle\cap\ker\partial_{k+1}=0$. Maximality is obtained when 
$\KK\langle K_j\rangle\oplus\ker{\partial_{k+1}}_{|\tilde{C}_{k+1}({\cF}^{k+1}_{j})}=\tilde{C}_{k+1}({\cF}^{k+1}_{j})$. 
Hence, the cardinality of $K_j$ equals $\dim\partial_{k+1}(\tilde{C}_{k+1}({\cF}^{k+1}_{j}))$.
The formula for $u_{k+1,j}$ follows since $\tilde{H}_{k+1}(\cF^{k+1}_j)=\ker{\partial_{k+1}}_{|\tilde{C}_{k+1}({\cF}^{k+1}_{j})}$.
Let us consider the long exact sequence for the pair $(\cF^{k+1}_{j},\cF^{k+1}_{j-1})$ and recall that the relative complex
vanishes outside index~$k+1$, i.e., $H_{k+1}(\cF^{k+1}_{j},\cF^{k+1}_{j-1})=\frac{\tilde{C}_{k+1}({\cF}^{k+1}_{j})}{\tilde{C}_{k+1}({\cF}^{k+1}_{j-1})}$ and the other homology groups vanish:
\[
\begin{tikzcd}[column sep=10pt]
0\arrow[r]&H_{k+1}(\cF^{k+1}_{j-1})\arrow[r]&H_{k+1}(\cF^{k+1}_{j})\arrow[r]&\frac{\tilde{C}_{k+1}({\cF}^{k+1}_{j})}{\tilde{C}_{k+1}({\cF}^{k+1}_{j-1})}
\arrow[r]&H_{k}(\cF^{k+1}_{j-1})\arrow[r]&H_{k}(\cF^{k+1}_{j})\arrow[r]&0
\end{tikzcd}
\]
Since the difference $u_{k+1,j}-u_{k+1,j-1}$ is the alternating sum of the first three terms, the result follows.
\end{proof}

\begin{lem}
\label{prop:lj}
Let $0\leq j\leq k+1$. There exists $L_j\subset\tilde{\cF}^{k}_{j}$
such that $\tilde{H}_k(M_{L_j})\to \tilde{H}_k(\cF)$ is injective,
it has the same image as
$\tilde{H}_k(\cF^{k}_j)\to \tilde{H}_k(\cF)$, and
$\tilde{H}_{k-1}(M_{L_j})\to \tilde{H}_{k-1}(\cF^{k}_{j})$
is an isomorphism. 
The cardinality of any such $L_j$ depends only on~$j$.
Its cardinality is
\[
u_{k,j}:=\dim\tilde{C}_{k}({\cF}^{k}_{j})-\dim\ker(\tilde{H}_{k}(\cF^{k}_j)\to \tilde{H}_k(\cF)).
\]
and
\begin{equation}\label{eq:lj}
u_{k,j}-u_{k,j-1}=\dim\tilde{H}_k(\cF^{k+1},\cF^{k}_{j-1})-\dim\tilde{H}_k(\cF^{k+1},\cF^{k}_{j}).
\end{equation}
\end{lem}

\begin{proof}
Let  $L_j\subset\tilde{\cF}^{k+1}_{j}$. The condition $\tilde{H}_{k}(M_{L_j})=0$ is equivalent to 
$\KK\langle L_j\rangle\cap\im\partial_{k+1}=0$ and maximality is obtained when 
$\KK\langle L_j\rangle\oplus(\im{\partial_{k+1}}\cap\tilde{C}_{k}({\cF}^{k}_{j}))=\tilde{C}_{k}({\cF}^{k}_{j})$
and the computation of $u_{k,j}$ follows.

Let us deal with the difference. As a first step, let us consider the long exact sequence for the pair $(\cF^{k+1},\cF^{k}_{j})$:
\[
\begin{tikzcd}[column sep=20pt]
0\arrow[r]&\tilde{H}_{k+1}(\cF^{k+1})\arrow[r]&\tilde{H}_{k+1}(\cF^{k+1},\cF^{k}_{j})\arrow[r]&\tilde{H}_{k}(\cF^{k}_{j})\arrow[r]&\tilde{H}_{k}(\cF^{k+1})\arrow[r]&0.
\end{tikzcd}
\]
Replacing the two terms by the kernel of this map, we obtain that 
\[
u_{k,j}:=\dim\tilde{C}_{k}({\cF}^{k}_{j})-\dim\tilde{H}_{k+1}(\cF^{k+1},\cF^{k}_j)+ \tilde{H}_{k+1}(\cF^{k+1}).
\]
Let us consider the long exact sequence for the triple $(\cF^{k+1},\cF^{k}_{j},\cF^{k}_{j-1})$:
\[
\begin{tikzcd}
0\arrow[r]&\tilde{H}_{k+1}(\cF^{k+1},\cF^{k+1}_{j-1})\arrow[r]&\tilde{H}_{k+1}(\cF^{k+1},\cF^{k+1}_{j})\arrow[r]&\frac{\tilde{C}_{k}({\cF}^{k}_{j})}{\tilde{C}_{k}({\cF}^{k}_{j-1})}
\arrow[r]&{}\\
{}\arrow[r]&\tilde{H}_{k}(\cF^{k+1},\cF^{k+1}_{j-1})\arrow[r]& \tilde{H}_{k}(\cF^{k+1},\cF^{k+1}_{j})\arrow[r]&0.
\end{tikzcd}
\]
The alternating sum yields the result.
\end{proof}

\begin{prop}\label{prop:minKL}
Let $(K,L)$ be a \emph{minimal} acyclic pair of size~$\rk\partial_{k+1}$. 
The simplices of~$K$ of weight~$\leq j$ form a subset $K_j$
as in{\rm~\autoref{prop:kj}} and the number of simplices of
weight equal to~$j$ is the right-hand side of~\eqref{eq:kj}.

The simplices of~$L$ of weight~$\leq j$ form a subset $L_j$
as in{\rm~\autoref{prop:lj}} and the number of simplices of
weight equal to~$j$ is the right-hand side of~\eqref{eq:lj}.

Any $(K,L)$ with the above properties is minimal.
\end{prop}

\begin{proof}
By \autoref{prop:acyclic}, such an acyclic pair can be constructed.
Let $(K,L)$ be an acyclic pair of size $\rk\partial_{k+1}$.
Let us denote by $v_{k+1,j}$ (resp. $v_{k,j}$) the number of simplices
in $K$ (resp. $L$) of weight~$\leq j$. Since $u_{k+1,k+2}=v_{k+1,k+2}$,
let $j$ be the first value for which $u_{k+1,j}=v_{k+1,j}$ but 
$u_{k+1,j-1}>v_{k+1,j-1}$. If such a $j$ exists it is easy to construct 
an acyclic pair $(K',L)$ where $\omega(K')<\omega(K)$. 

A similar argument works for~$L$. Then, if $(K,L)$ is minimal, no such $j$ exists, 
i.e., $K,L$ satisfy the conditions of the statement.
%
%
%
\end{proof}

%
%

Minimal acyclic pairs are important since they determine their corresponding Fitting ideals.
The following result gives a formula for the maximal multiplicity of $s=(t-\lambda)$ in the Fitting ideals
(see~\eqref{eq:bs} for the definition of $b_{k,j}$).

\begin{thm}
\label{thm:bkl}
There are $\rk\partial_{k+1}=r+\ell_k$ non-zero Fitting ideals. Then
\begin{gather*}
\sum_j j\cdot \n_{k,j}(d)=b_{k,\ell_k}=
\sum_{j=0}^{k+1}\dim\tilde{H}_k(\cF^{k+1}_{j})-(k+2)\dim\tilde{H}_k(\cF)+\\
\sum_{j=0}^k \dim\tilde{H}_{k+1}(\cF^{k+1},\cF^{k}_{j})
-(k+1)\dim\tilde{H}_{k+1}(\cF^{k+1},\cF^k).
\end{gather*}
\end{thm}

\begin{proof}
From \autoref{lem:mult_KL} we need to compute $\omega(K)+\omega(L)-\omega(\cF^k)$ for a minimal acyclic pair $(K,L)$
of size $\rk\partial_{k+1}$. 

From \autoref{prop:minKL} and \autoref{prop:kj}, we have the following value for $\omega(K)$:
\begin{gather*}
\sum_{j=1}^{k+2} j\dim \tilde{H}_k(\cF^{k+1}_{j-1})- \sum_{j=1}^{k+2} j\dim \tilde{H}_k(\cF^{k+1}_{j})=
\sum_{j=0}^{k+1} \dim \tilde{H}_k(\cF^{k+1}_{j})-(k+2)\dim \tilde{H}_k(\cF^{k+1}).
\end{gather*}
From \autoref{prop:minKL} and \autoref{prop:lj}, we have the following value for $\omega(L)$:
\begin{gather*}
\sum_{j=1}^{k+1} j\dim \tilde{H}_k(\cF^{k+1},\cF^{k}_{j-1})- \sum_{j=1}^{k+1} j\dim \tilde{H}_k(\cF^{k+1},\cF^{k}_{j})=\\
\sum_{j=0}^{k} \dim \tilde{H}_k(\cF^{k+1},\cF^{k}_{j})-(k+1)\dim \tilde{H}_k(\cF^{k+1},\cF^k)=\sum_{j=0}^{k} \dim \tilde{H}_k(\cF^{k+1},\cF^{k}_{j}).
\end{gather*}
Finally, note that $\dim\tilde{C}_{k}(\cF^k,\cF^k_j)$ is the number of $k$-cells of weight $>j$. Hence
\begin{equation*}
\omega(\cF^k)=\sum_{j=0}^k\dim\tilde{C}_{k}(\cF^k,\cF^k_j)=\sum_{j=0}^k\dim\tilde{H}_{k}(\cF^k,\cF^k_j). 
\end{equation*}
Let us consider the triple $(\cF^{k+1},\cF^k,\cF^k_j)$ and its homology long exact sequence: 
\[
\begin{tikzcd}[column sep=10pt]
0\ar[r]&\tilde{H}_{k+1}(\cF^{k+1},\cF^k_{j})\ar[r]&
\tilde{H}_{k+1}(\cF^{k+1},\cF^k)\ar[r]&
\tilde{H}_{k}(\cF^{k},\cF^k_{j})\ar[r]&
\tilde{H}_{k}(\cF^{k+1},\cF^k_{j})\ar[r]&0
\end{tikzcd}
\]
We obtain that
\begin{gather*}
\omega(L)-\omega(\cF^k)=\sum_{j=0}^k \dim\tilde{H}_{k+1}(\cF^{k+1},\cF^k_{j}) -(k+1) \dim\tilde{H}_{k+1}(\cF^{k+1},\cF^k).
\qedhere
\end{gather*}
\end{proof}

\subsection{Bounding the Jordan blocks}
\mbox{}

The purpose of this section is to give a bound for the size of the Jordan blocks in the 
$(t-\lambda)$-primary part of ${H}_{k+1}(\bt^\chi;\KK)$.
This will be done by computing the drop in multiplicities in the Fitting ideals. 
This is measured by the acyclic pairs $(K,L)$ as considered above. Their construction is studied next.

\begin{prop}
Let $(K,L)$ be an acyclic pair with size $<\rk\partial_{k+1}$.
Let $\sigma\in\tilde{\cF}^{k+1}\setminus L$. Then, there exists
$\tau\in L$ such that $(K\cup\{\sigma\},L\setminus\{\tau\})$
is an acyclic pair.
\end{prop}

\begin{proof}
Note that $M_{K\cup\{\sigma\}}=M_{K}\cup|\sigma|$ and
$\partial{|\sigma|}=M_{K}\cap|\sigma|$. Using Mayer-Vi\'etoris
long exact sequence we have
\[
\begin{tikzcd}
\tilde{H}_{k+1}(M_{K\cup\{\sigma\}})\ar[r]&[-5mm]
\tilde{H}_k(\partial{|\sigma|})\ar[r]&[-5mm]
\tilde{H}_k(M_{K})\oplus\tilde{H}_k(|\sigma|)\ar[r]&[-5mm]
\tilde{H}_k(M_{K\cup\{\sigma\}})\ar[r]&[-5mm]
\tilde{H}_{k-1}(\partial{|\sigma|})
\\[-25pt]
\rotatebox{90}{$=$}&
\rotatebox{-90}{\hspace{-3mm}$\cong$}&
\hspace{18mm}\rotatebox{90}{$=$}&&
\rotatebox{90}{$=$}
\\[-25pt]
0&
\bz&
\hspace{18mm}0&&0
\end{tikzcd}
\]
i.e, 
\[
\begin{tikzcd}
0\ar[r]&\bz\ar[r]&\tilde{H}_k(M_{K})\ar[r]&\tilde{H}_k(M_{K\cup\{\sigma\}})\ar[r]&0.
\end{tikzcd}
\]
In particular, the composition $\phi:\tilde{H}_k(M_L)\to\tilde{H}_k(M_{K\cup\{\sigma\}})$ has a one-dimensional kernel. Recall that 
$\tilde{H}_k(M_{K\cup\{\sigma\}})=\ker{\alpha_k}_{|\KK\langle L\rangle}$.
Let $\tau\in L$ in the support of the non-zero elements of the kernel of
$\phi$. We have:
\[
\begin{tikzcd}
0&\hspace{18mm}0&&
\bz\\[-25pt]
\rotatebox{90}{$=$}&
\hspace{18mm}\rotatebox{90}{$=$}&&
\rotatebox{-90}{\hspace{-3mm}$\cong$}
\\[-25pt]
\tilde{H}_{k}(\partial|\tau|)\ar[r]&
\tilde{H}_k(M_{L\setminus\{\tau\}})\oplus\tilde{H}_k(|\tau|)\ar[r]&
\tilde{H}_k(M_{L})\ar[r]&
\tilde{H}_{k-1}(\partial{|\tau|})
\\[-15pt]
\tilde{H}_{k-1}(\partial{|\tau|})\ar[r]&
\tilde{H}_{k-1}(M_{L\setminus\{\tau\}})\oplus\tilde{H}_{k-1}(|\tau|)\ar[r]&
\tilde{H}_{k-1}(M_{L})\ar[r]&
\tilde{H}_{k-2}(\partial{|\tau|})
\\[-25pt]
\rotatebox{-90}{\hspace{-3mm}$\cong$}&
\hspace{18mm}\rotatebox{90}{$=$}&&
\rotatebox{90}{$=$}
\\[-25pt]
\bz&
\hspace{18mm}0&&0
\end{tikzcd}
\]
i.e,
\[
\begin{tikzcd}
0\ar[r]&[-12pt]\tilde{H}_k(M_{L\setminus\{\tau\}})\ar[r]&[-5pt]
\tilde{H}_k(M_{L})\ar[r]&[-5pt]
\bz\ar[r]&[-5pt]
\tilde{H}_{k-1}(M_{L\setminus\{\tau\}})\ar[r]&[-5pt]
\tilde{H}_{k-1}(M_{L})\ar[r]&[-12pt]0.
\end{tikzcd}
\]
For the choice of $\tau$ the map $\bz\to\tilde{H}_{k-1}(M_{L\setminus\{\tau\}})$ is zero and the first terms yield a short exact sequence. Moreover, by construction, the map 
$\tilde{H}_{k}(M_{L\setminus\{\tau\}})\to\tilde{H}_{k}(M_{K\cup\{\sigma\}})$ is injective; as the dimensions are equal,
it is bijective and $(K\cup\{\sigma\},L\setminus\{\tau\})$ is an acyclic pair.
\end{proof}

\begin{cor}
Any acyclic pair~$(K',L')$ is obtained from an acyclic pair~$(K,L)$ of size $\rk\partial_{k+1}$ by removing simplices from~$K$
and adding simplices to~$L$.
\end{cor}

The following result provides information about acyclic pairs after adding simplices, rather than removing them.

\begin{prop}\label{prop:less-size}
Let $(K,L)$ be an acyclic pair of positive size. Let $\sigma\in K$ and let $\tau\in\tilde{\cF}^{k}\setminus L$.
Then $(K\setminus\{\sigma\},L\cup\{\tau\})$ is an acyclic pair if and only if
$\partial_{k+1}(\sigma)$ and $\tau$ are both non trivial
$\bmod \partial_{k+1}(\KK\langle K\setminus\{\sigma\}\rangle)\oplus\KK\langle L\rangle$.
\end{prop}

\begin{proof}
Since $(K,L)$ is an acyclic pair, we know that $\partial_{k+1}(\KK\langle K\rangle)\oplus\KK\langle L\rangle=
\tilde{C}_k(\cF)$ and ${\partial_{k+1}}_{|\KK\langle K\rangle}$ is injective. Given $\sigma,\tau$
as in the statement, we clearly have that the pair
$(K\setminus\{\sigma\},L\cup\{\tau\})$ is acyclic if and only if $\partial_{k+1}(\KK\langle K\setminus\{\sigma\}\rangle)\oplus\KK\langle L\cup\{\tau\}\rangle=
\tilde{C}_k(\cF)$.

Let $H:=\partial_{k+1}(\KK\langle K\setminus\{\sigma\}\rangle)\oplus\KK\langle L\rangle$.
Since $\partial_{k+1}(\KK\langle K\rangle)\oplus\KK\langle L\rangle=
\tilde{C}_k(\cF)$ we deduce that both $\partial_{k+1}(\sigma)$ and $\tau$ are non trivial $\bmod H$.
\end{proof}

This provides an upper bound for the Jordan blocks of the $(t-\lambda)$-primary part of the 
homology~${H}_{k+1}(\bt^\chi;\KK)$.

\begin{cor}\label{cor:Jordan_max}
The value $a_{k,\ell_k}$ is at most $k+2$, i.e., any exponent in the torsion of $H_{k+1}(\bt^\chi)$ 
is at most~$k+2$.
\end{cor}

\begin{proof}
We have to find minimal acyclic pairs of size~$\rk\partial_{k+1}-1$. It must come from a minimal
acyclic pair of size~$\rk\partial_{k+1}$. The highest gap is to take out a $(k+1)$-simplex~$\sigma$ of weight~$k+2$
and introduce a $k$-simplex~$\tau$ of weight~$0$.
\end{proof}

The number of blocks of size $(k+2)$ is given in terms of the flag complex.

\begin{prop}
\label{prop:maxJordan}
The number $\#\{j\geq 1\mid a_{k,j}=k+2\}$ 
equals the rank of the map from 
$\ker(\tilde{H}_k(\cF^k_0)\to\tilde{H}_k(\cF))$ to $\ker(\tilde{H}_k(\cF^{k+1}_{k+1})\to\tilde{H}_k(\cF))$.
\end{prop}

\begin{proof}
Let us fix an arbitrary minimal acyclic pair $(K,L)$ of size~$\rk\partial_{k+1}$.
The set of $k$-simplices of weight~$0$ not in~$L$ is a natural basis
of $H_k(\cF^k_0,M_{L_0})$. In the same way, the set of $(k+1)$-simplices of weight~$k+2$ in~$K$ is a natural basis
of $H_k(M_K,M_{K_{k+1}})$. We are going to use long exact sequences of pairs for $(\cF^k_0,M_{L_0})$. For the first one:
\[
\begin{tikzcd}[column sep=7pt]
H_{k+1}(\cF^k_0,M_{L_0})\ar[r]&
H_{k}(M_{L_0})\ar[r]&
H_{k}(\cF^k_0)\ar[r]&
H_{k}(\cF^k_0,M_{L_0})\ar[r]&
H_{k-1}(M_{L_0})\ar[r,"\cong"]&[3pt]
H_{k-1}(\cF^k_0).\\[-25pt]
\rotatebox{90}{$=$}\\[-25pt]
0
\end{tikzcd}
\]
We obtain a short exact sequence. Recall that $H_{k}(M_{L_0})\to H_{k}(\cF)$ is injective and
its image coincides with the one of $H_{k}(\cF^k_0)\to H_{k}(\cF)$. 
We deduce a natural isomorphism 
\[
\coker(H_{k}(M_{L_0})\to H_{k}(\cF^k_0))\cong\ker(H_{k}(\cF^k_0)\to H_{k}(\cF)).
\] 
Hence $\ker(H_{k}(\cF^k_0)\to H_{k}(\cF))\to H_{k}(\cF^k_0,M_{L_0})$
is an isomorphism.

A similar argument shows that $H_{k+1}(K,K_{k+1})\to\ker(H_k(\cF^{k+1}_{k+1})\to H_k(\cF))$ is an isomorphism.
The result follows by \autoref{prop:less-size}.
\end{proof}

With the same ideas we can compute the maximal size of the exponents
and give a bound for its number.
Let us denote by 
\[
c_{k,i,j}:=\rk\left(\ker(\tilde{H}_k(\cF^k_j)\to\tilde{H}_k(\cF))\longrightarrow 
\ker(\tilde{H}_k(\cF^{k+1}_{i-1})\to\tilde{H}_k(\cF))\right),\quad i>j.
\]

\begin{prop}\label{prop:maxJordan1}
Assume that $\rk\partial_{k+1}>1$ and define $a_{k,0}=0$ if $\ell_k=1$. Then, $a_{k,\ell_k}$ is
equal to the maximum value $a$ such that $\exists i,j$, $j-i=a$, and $c_{k,i,j}\neq 0$.
Moreover the number of terms $a_{k,h}$ with the same value is at least
$\max\{a_{k,i,j}\mid i-j=a\}$.
\end{prop}
In particular, the semisimple case can be characterized when this value equals~$1$.

\section{A reduction to even characters}\label{sec:even}
The purpose of this section is twofold. First we present a special type of characters called 
\emph{even characters} and then we show that the $\Phi_d$-torsion part of $H_*(A_\Gamma^\chi;\KK)$
can be described as the $(t+1)$-torsion part of an even character.

A character $\rho:G_\Gamma\to\bz$ such that $\{\rho(v)\mid v\in\V\}=\{1,2\}$ will be called
\emph{even character}.

\subsection{Even characters associated with a pair \texorpdfstring{$(\chi,d)$}{(chi,d)}}
\mbox{}

Let $\chi:G_\Gamma\to\bz$ be a surjective character. Given any $d>1$ one can associate an even 
character to $\chi$ and $d$ as follows $\rho_{\chi,d}:G_\Gamma\to\bz$, $\rho_{\chi,d}(v):=2^{\omega_d(v)}$.
Since $\chi$ is surjective and $d>1$, one has the following equality 
for the set of indices $\{\rho_{\chi,d}(v)\}_{v\in V}=\{1,2\}$. In particular, $\rho_{\chi,d}$ is 
surjective and one can consider the exponents $\n_{k,j}(2)$ for $\rho_{\chi,d}$. The following result
is straightforward.

\begin{cor}
The exponents $\n_{k,j}(d)$ for $\chi$ coincide with the exponents $\n_{k,j}(2)$ for $\rho_{\chi,d}$.
\end{cor}

\subsection{The \texorpdfstring{$(t+1)$}{(t+1)}-primary part of an even character}
\mbox{}

In this section we fix a surjective \emph{weight map} $\omega:\V\to\{0,1\}$, defining
an even character $\rho:G_\Gamma\to\bz$, defined as $\rho(v):=2^{\omega(v)}$. We denote
also by $\omega$ its extension to $\cF$, where $\omega(\sigma):=\sum_{v\in\sigma}\omega(v)$.
The goal is to compute the $(t+1)$-primary part of $H_{k+1}(\bt^\rho;\KK)$. The double cover 
$\bt^{\rho_2}$ of $\bt$ defined by the even character $\rho$ detects $\ell_k$, the number of 
summands in the $(t+1)$-primary part as defined in~\eqref{eq:bs}.

\begin{thm}
\label{thm:lk}
The number $\ell_k$ coincides with 
$$
\dim\!{H}_{k+1}(\bt^{\rho_2};\mathbb{K}) -
\dim\tilde{H}_k(\cF)-\dim\tilde{C}_k(\cF)
-\ell_{k-1}\!=\!
\dim{H}_{k+1}^-(\bt^{\rho_2};\mathbb{K}) -
\dim\tilde{H}_k(\cF)
-\ell_{k-1},
$$
where ${H}_{k+1}^-(\bt^{\rho_2};\mathbb{K})$ is the anti-invariant part of the homology of $\bt^{\rho_2}$.
\end{thm}

\begin{proof}
We recover ideas from Sakuma's formula~\cite{Sakuma-homology} for the homology of unramified coverings.
The character $\rho_2:G_\Gamma\to\bz/2$, where $\rho_2(g):=\rho(g)\mod 2$ determines an unramified 
double covering of~$\bt$ denoted $\tilde{\rho}_2:\bt^{\rho_2}\to\bt$. For the sake of brevity we denote
$\Lambda=\mathbb{K}[t^{\pm 1}]$, $\Lambda_+:=\Lambda/\langle t-1\rangle\cong\mathbb{K}$ and 
$\Lambda_-:=\Lambda/\langle t+1\rangle\cong\mathbb{K}$. Recall that $\Lambda\otimes\Lambda_-=\Lambda_-$ 
and that if $\lambda\in\mathbb{K}$ and $n>0$ then 
$\Lambda_-\otimes \Lambda/\langle (t-\lambda)^n\rangle=0$ if $\lambda\neq -1$ 
and $\Lambda_-\otimes \Lambda/\langle (t+1)^n\rangle=\Lambda_-$: 
\begin{gather*}
\ell_k
=
\dim{H}_{k+1}(\bt^{\rho_2};\mathbb{K})\otimes_\Lambda \Lambda_- -r_k.
\end{gather*}
From~\ref{PS1}, $r_k=\dim\tilde{H}_k(\cF)$ and by definition 
${H}_{k+1}(\bt^{\rho_2};\mathbb{K})={H}_{k+1}(C_*(\bt^{\rho_2}))$.
By the Universal Coefficient Theorem
\begin{gather*}
\dim{H}_{k+1}(C_*(\bt^{\rho_2}))\otimes_\Lambda \Lambda_- =
\dim{H}_{k+1}(C_*(\bt^{\rho_2})\otimes_\Lambda \Lambda_-) - \dim\tor_1({H}_{k}(C_*(\bt^{\rho_2})),\Lambda_-).
\end{gather*}
By the properties of the $\tor$ functor $\dim\tor_1({H}_{k}(C_*(\bt^{\rho_2})),\Lambda_-)=\ell_{k-1}$.
The homology of $\bt^{\rho_2}$ decomposes in the invariant and anti-invariant part for the monodromy of the covering:
\begin{gather*}
\dim{H}_{k+1}(C_*(\bt^{\rho_2})\otimes_\Lambda \Lambda_-)=
\dim{H}_{k+1}(\bt^{\rho_2};\mathbb{K}) -
\dim{H}_{k+1}(C_*(\bt^{\rho_2})\otimes_\Lambda \Lambda_+),
\end{gather*}
where ${H}_{k+1}(C_*(\bt^{\rho_2})\otimes_\Lambda \Lambda_+)=\tilde{C}_k(\cF)$.
Summarizing,
\begin{gather*}
\ell_k=\dim{H}_{k+1}(\bt^{\rho_2};\mathbb{K}) -
\dim\tilde{H}_k(\cF)-\dim\tilde{C}_k(\cF)
-\ell_{k-1}.
\qedhere
\end{gather*}
\end{proof}

\begin{rem}
From the computational point of view ${H}_{*+1}^-(\bt^{\rho_2};\mathbb{K})$, is the homology
of the complex $C_*(\cF^\rho,\KK)$ after the evaluation $t=-1$, i.e., tensoring by $\Lambda_-$.
\end{rem}

\subsection{An application to even characters}
\mbox{}

The above results for $H_1(\bt^\rho,\KK)$ for an even character~$\rho$ give the following.
They depend only on $\Gamma$. Let us denote by $\V_0=\cF_0^0$ the set of vertices of weight~$0$.
We denote also by $\Gamma_0=\cF_0^1$ (resp. $\Gamma_1=\cF_1^1$) the graph whose set of vertices is $\V$
and the set of edges consist of the edges of $\Gamma$ with weight $0$ (resp. $\leq 1$).

\begin{prop}
The dimension of the $(t+1)$-primary part of $H_1(\bt^\rho,\KK)$ equals
\begin{gather*}
\dim\tilde{H}_0(\Gamma_{0})+\dim\tilde{H}_0(\Gamma_{0})-2\dim\tilde{H}_0(\Gamma)+
\dim\tilde{H}_{0}(\Gamma,V_{0})
-\dim\tilde{H}_{0}(V,V_0).
\end{gather*}
The number of non-semisimple factors (exponent~$2$) is the rank of the map
\[
\begin{tikzcd}
\ker(\tilde{H}_0(\V_0)\to\tilde{H}_0(\Gamma))\ar[r]&\tilde{H}_0(\Gamma_{1}).
\end{tikzcd}
\]
\end{prop}

\begin{cor}
\label{cor:H1}
Assume $\Gamma$ is connected. Then the dimension of the $(t+1)$-primary part of $H_1(\bt^\rho,\KK)$ 
equals
\[
\dim{H}_0(\Gamma_0)+\dim{H}_0(\Gamma_1)-2-\omega(\V).\quad
\]
and the number of non-semisimple blocks is $\rk(\tilde{H}_0(\V_0)\to\tilde{H}_0(\Gamma_{1}))$.
\end{cor}

\section{Main result}
\label{sec:main}
Let 
\[
H_{k+1}(A^\chi_\Gamma;\KK)\cong\KK[t^{\pm 1}]^{\n_k}
\bigoplus_{d\in\mathbb{Z}_{>0}}\bigoplus_{j>0}
\left(\KK[t^{\pm 1}]/\langle\Phi_d^j\rangle\right)^{\n_{k,j}(d)}
\]
be the decomposition of this finitely generated $\KK[t^{\pm 1}]$-module.
If one assumes certain acyclicity conditions on the flag complex $\cF$, then  
a more concise description of $H_{k+1}(A^\chi_\Gamma;\KK)$ can be given. 
Consider $d>1$ and the weighted graph obtained from $\Gamma$ by attaching the weight
$$\omega_d(v)=\begin{cases} 1 & \textrm{ if } d|n_v\\ 0 & \textrm{ otherwise,}\end{cases}$$ 
as defined in \autoref{sec:torsion}. This weight defines a filtration in the flag complex
$\cF_j^k$ as defined in~\eqref{eq:Fdef}. Note that the filtration $\cF_j^k$ depends on $d$, 
we drop $d$ to avoid cumbersome notation.

We denote by $\tilde h_i(\cF^m_j):=\dim_\KK \tilde H_i(\cF^m_j;\KK)$ and by 
$\omega(\tilde \cF^k)$ the weight of the set of $k$-simplices of $\cF$ as defined in
\autoref{sec:fitting}.

\begin{thm}
\label{thm:main}
Assume $\cF$ is $k$-acyclic, then
\[
H_{k+1}(A^\chi_\Gamma;\KK)\cong\KK[t^{\pm 1}]^{\tilde h_k(\cF)}\oplus
\left(\frac{\KK[t^{\pm 1}]}{\langle (t-1)\rangle}\right)^{\rk \partial_{k+1}}
\bigoplus_{d\in\mathbb{Z}_{>0}}
\left(
\bigoplus_{j=1}^{k+2}
\left(\frac{\KK[t^{\pm 1}]}{\langle\Phi_d^j\rangle}\right)^{\n_{k,j}(d)}
\right),
\]
where
$$
\begin{aligned}
\sum_{j=1}^{k+2} j\cdot \n_{k,j}(d)=&
\sum_{i=k}^{k+1} \sum_{j=0}^i\tilde h_{k}(\cF^{i}_j)-(k+1)\tilde{h}_k(\cF^k),\\
\ell_k(d)=\sum_{j=1}^{k+2} \n_{k,j}(d)= &
\sum_{i=0}^{k}(-1)^{k-i}\dim{H}_{i+1}^-(\bt^{\rho_2};\mathbb{K})
\\
\n_{k,k+2}(d) =& \rk\left( \tilde{H}_k(\cF^k_0)\to \tilde{H}_k(\cF^{k+1}_{k+1})\right).
\end{aligned}
$$
The maximal $h$ such that $r_{k,h}>0$ is 
\[
\max\left\{i-j+1\mid\rk\left(\tilde{H}_k(\cF^k_j)\longrightarrow 
\tilde{H}_k(\cF^{k+1}_{i-1})\right)>0\right\}.
\]

This determines inductively the $\KK[t^{\pm 1}]$-module structure of $H_1(A^\chi_\Gamma;\KK)$
and $H_2(A^\chi_\Gamma;\KK)$ completely.
\end{thm}

\begin{proof}
The first formula follows from \autoref{thm:bkl} since $\tilde h_k(\cF)=0$ by hypothesis, 
$\tilde h_k(\cF^{k+1},\cF_j^k)=\tilde h_{k-1}(\cF_j^k)$ using the exact sequence of pairs
and the acyclicity of $\cF$, and $\sum_{j=0}^k\dim\tilde{C}_{k}(\cF^k,\cF^k_j)=\omega(\tilde \cF^k)$.
The second formula follows \autoref{thm:lk}
and the third equality can be obtained from \autoref{prop:maxJordan} and the acyclicity of~$\cF$.
The maximality statement follows \autoref{prop:maxJordan1}.
\end{proof}

\section{Examples}
\label{sec:examples}

In this section we present computations for a variety of examples. An interested reader may 
apply the results to other examples and check the computations below using the notebook in
\url{https://github.com/enriqueartal/ArtinKernels}, and using 
\texttt{Sagemath}~\cite{sage90} or \texttt{Binder}~\cite{binder}.

\begin{example}
Let $\Gamma$ be the graph in \autoref{fig:tree}. 
The labels on the vertices correspond to the values $\rho_{\chi,d}(v)$.
\begin{figure}[ht]
\begin{center}
\begin{tikzpicture}
\Vertex[style={minimum size=1.0cm,shape=circle},LabelOut=false,L=\hbox{$18$},x=2.0cm,y=5.0cm]{v0}
\Vertex[style={minimum size=1.0cm,shape=circle},LabelOut=false,L=\hbox{$4$},x=0.0cm,y=5.0cm]{v1}
\Vertex[style={minimum size=1.0cm,shape=circle},LabelOut=false,L=\hbox{$12$},x=4.0cm,y=5.0cm]{v2}
\Vertex[style={minimum size=1.0cm,shape=circle},LabelOut=false,L=\hbox{$9$},x=6.0cm,y=5.0cm]{v3}
\Edge[lw=0.05cm](v0)(v1)
\Edge[lw=0.05cm](v0)(v2)
\Edge[lw=0.05cm](v2)(v3)
\end{tikzpicture}
\end{center}
\caption{A linear tree}
\label{fig:tree}
\end{figure}
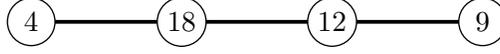
The structure of $H_1(A^\chi_\Gamma;\mathbb{C})$ is:
$$
H_1(A^\chi_\Gamma;\KK)=
\left(\frac{\KK[t^{\pm 1}]}{\Phi_1}\right)^{3}\oplus
\left(\frac{\KK[t^{\pm 1}]}{\Phi_2\Phi_3}\right)^{2}\oplus
\left(\frac{\KK[t^{\pm 1}]}{\Phi_6^2\Phi_4\Phi_9\Phi_{12}\Phi_{18}}\right).
$$
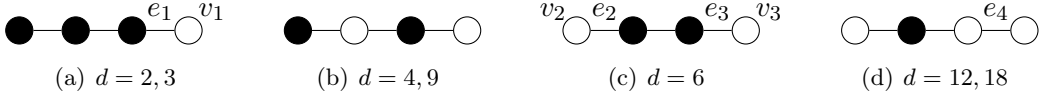
\begin{figure}[ht]
\begin{center}
\subfigure[$d=2,3$\label{fig:tree2a}]{
\begin{tikzpicture}[scale=.75]
\node [draw,fill=black,circle] (a) at (0,0) {};
\node [draw,fill=black,circle] (b) at (1,0) {};
\node [draw,fill=black,circle] (c) at (2,0) {};
\node [draw,circle] (d) at (3,0) {};
\draw[] (a) edge (b);
\draw[] (b) edge (c);
\draw[] (c) edge node[above,pos=.5] {$e_1$} (d);
\node[above right] at (d) {$v_1$};
\end{tikzpicture}}
\hfil
\subfigure[$d=4,9$\label{fig:tree2b}]{
\begin{tikzpicture}[scale=.75]
\node [draw,fill=black,circle] (a) at (0,0) {};
\node [draw,circle] (b) at (1,0) {};
\node [draw,fill=black,circle] (c) at (2,0) {};
\node [draw,circle] (d) at (3,0) {};
\draw[] (a) edge (b);
\draw[] (b) edge (c);
\draw[] (c) edge (d);
\end{tikzpicture}}
\hfil
\subfigure[$d=6$\label{fig:tree2c}]{
\begin{tikzpicture}[scale=.75]
\node [draw,circle] (a) at (0,0) {};
\node [draw,fill=black,circle] (b) at (1,0) {};
\node [draw,fill=black,circle] (c) at (2,0) {};
\node [draw,circle] (d) at (3,0) {};
\draw[] (a) edge node[above,pos=.5] {$e_2$} (b);
\draw[] (b) edge (c);
\draw[] (c) edge node[above,pos=.5] {$e_3$} (d);
\node[above left] at (a) {$v_2$};
\node[above right] at (d) {$v_3$};
\end{tikzpicture}}
\hfil
\subfigure[$d=12,18$\label{fig:tree2d}]{
\begin{tikzpicture}[scale=.75]
\node [draw,circle] (a) at (0,0) {};
\node [draw,fill=black,circle] (b) at (1,0) {};
\node [draw,circle] (c) at (2,0) {};
\node [draw,circle] (d) at (3,0) {};
\draw[] (a) edge (b);
\draw[] (b) edge (c);
\draw[] (c) edge node[above,pos=.5] {$e_4$} (d);
\end{tikzpicture}}
\end{center}
\caption{Weighted graphs. Weight-one vertices in black.}
\label{fig:tree2}
\end{figure}
Let us study each one of the $\Phi_d$-primary factors, for $d=1,2,3,4,6,9,12,18$.
\begin{itemize}

\item $d=1$: Recall that $\n_{0,1}(1)=\rk\partial_1$. The map
$\partial_1$ is injective as $\Gamma$ is a tree; hence
$\n_{0,1}(1)=3$.

\item $d=2,3$: The primary parts for $\Phi_2$ and 
$\Phi_3$ are obtained considering the associated even character $\rho$
as in Figure~\ref{fig:tree2a}, up to automorphism of the graph. We have:
\[
\Gamma_0=V,\quad \Gamma_1=V\cup\text e_1,\quad V_0=\{v_1\}.
\]
The dimension of the $(t+1)$-primary part for this~$\rho$ is 
$\tilde{h_0}(\Gamma_1)+\tilde{h_0}(\Gamma_0)-\omega(V)=2+3-3=2$, 
according to \autoref{cor:H1}. 
Since $\V_0$ is connected this part is semisimple and hence the primary parts for $\Phi_2$ and 
$\Phi_3$ are $(\KK[t^{\pm 1}]/\Phi_2)^2$ and $(\KK[t^{\pm 1}]/\Phi_3)^2$ respectively.

\item $d=4,9$: The associated even character~$\rho$ is described in
Figure~\ref{fig:tree2b}. We have:
\[
\Gamma_0=V,\quad \Gamma_1=\Gamma.
\]
The dimension of the $(t+1)$-primary part for~$\rho$ is $4+1-2-2=1$.
This gives the result for $\Phi_4$ and~$\Phi_9$.

\item $d=6$: The associated even character~$\rho$ is described in
Figure~\ref{fig:tree2c}. We have:
\[
\Gamma_0=V,\quad \Gamma_1=e_2\cup e_3,\quad V_0=\{v_2,v_3\}.
\]
The dimension of the corresponding primary part is $4+2-2-2=2$. In this case $\V_0$ consists of the end vertices
and the map $\tilde{H}_0(\V_0)\to\tilde{H}_0(\Gamma_{1})$ is bijective. Hence the primary part for $\Phi_6$ is 
$\KK[t^{\pm 1}]/\Phi_6^2$.

\item $d=12,18$: The associated even character~$\rho$ is described in
Figure~\ref{fig:tree2d}. We have:
\[
\Gamma_0=V\cup e_4,\quad \Gamma_1=\Gamma.
\]
The dimension of the corresponding primary part is $3+1-2-1=1$. The computation ends.
\end{itemize}
\end{example}

\begin{example}
This example shows that the non-vanishing restrictions $\chi(v)=n_v\neq 0$ 
on the character $\chi$ are necessary.
Let us consider the same graph in \autoref{fig:tree} where the values of $n_v$'s are $(1,0,2,2)$. The matrices for $\partial^\chi_j$ are
\[
j=1:
\begin{pmatrix}
t-1&0&t^2-1&t^2-1 
\end{pmatrix}\quad 
j=2:
\begin{pmatrix}
0&0&0\\
t-1&1-t^2&0\\
0&0&1-t^2\\
0&0&t^2-1
\end{pmatrix}
\]
Hence $H_0(\bt^\chi;\KK)=\KK[t^{\pm 1}]/\langle t-1\rangle\cong\KK$.
If we denote by $v_j$ the vertices, then $\ker\partial^\chi_1$ 
is the free-module generated by $v_2,v_4-v_3,(t+1)v_1-v_3$.
Since the image of $\partial^\chi_2$ is the free module generated
by $(t-1)v_2$ and $(t^2-1)(v_4-v_3)$ we have that 
$H_1(\bt^\chi;\KK)=\KK[t^{\pm 1}]\oplus(\KK[t^{\pm 1}]/\langle t-1\rangle)^2\oplus\KK[t^{\pm 1}]/\langle t+1\rangle$. 
In particular, it does not satisfy \autoref{thm:main} since $\tilde h_0(\cF)=0$, but $H_1(\bt^\chi;\KK)$ 
is not a torsion module.
\end{example}

\begin{example}
We consider the graph $\Gamma$ in \autoref{fig:kite} and the even character $\rho$ defined by its labels.
The flag complex~$\cF$ is obtained by adding the triangle.

\begin{figure}[ht]
\begin{center}
\begin{tikzpicture}
\Vertex[style={minimum size=1.0cm,shape=circle},LabelOut=false,L=\hbox{$2$},x=0.0cm,y=2.5cm]{v0}
\Vertex[style={minimum size=1.0cm,shape=circle},LabelOut=false,L=\hbox{$2$},x=-2.0cm,y=0.0cm]{v1}
\Vertex[style={minimum size=1.0cm,shape=circle},LabelOut=false,L=\hbox{$2$},x=2.0cm,y=0.0cm]{v2}
\Vertex[style={minimum size=1.0cm,shape=circle},LabelOut=false,L=\hbox{$1$},x=0.0cm,y=4.5cm]{v3}
\Vertex[style={minimum size=1.0cm,shape=circle},LabelOut=false,L=\hbox{$1$},x=-4.0cm,y=0.0cm]{v4}
\Vertex[style={minimum size=1.0cm,shape=circle},LabelOut=false,L=\hbox{$1$},x=4.0cm,y=0.0cm]{v5}
\Edge[lw=0.05cm](v0)(v1)
\Edge[lw=0.05cm](v0)(v2)
\Edge[lw=0.05cm](v1)(v2)
\Edge[lw=0.05cm](v0)(v3)
\Edge[lw=0.05cm](v1)(v4)
\Edge[lw=0.05cm](v2)(v5)
\node[left=6pt] at (v3) {$v_3$};
\node[above=8pt] at (v4) {$v_4$};
\node[above=8pt] at (v5) {$v_5$};
\node[above] at ($.5*(v4)+.5*(v1)$) {$e_4$};
\node[above] at ($.5*(v5)+.5*(v2)$) {$e_5$};
\node[left] at ($.5*(v3)+.5*(v0)$) {$e_3$};
\end{tikzpicture}
\caption{Kite graph}
\label{fig:kite}
\end{center}
\vspace{-3mm}
\end{figure}
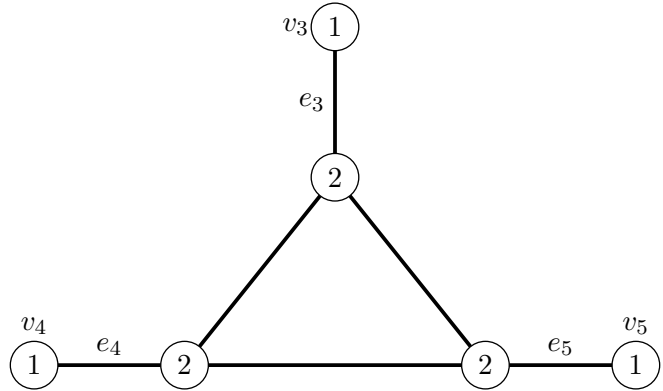
We have
\[
\Gamma_0=V,\quad\V_0=\{v_3,v_4,v_5\},\quad \Gamma_1=e_3\cup e_4\cup e_5,\quad 
\cF^2_j=\Gamma,\ j=0,1,2.
\]
We deduce that the free part of the homology vanishes. 

Since $\rk\partial_1=5$, the  dimension of the $(t-1)$-primary part for $H_1(A^\rho;\KK)$ is~$5$ and it is semisimple.
Using \autoref{cor:H1}, the dimension of the $(t+1)$-primary part for $H_1(A^\rho;\KK)$ is $6+3-2-3=4$. 
Also, since the map $\tilde{H}_0(\V_0)\to\tilde{H}_0(\Gamma_1)$ is an isomorphism of vector spaces of dimension~$2$. 
Hence, the $(t+1)$-primary part for $H_1(A^\rho;\KK)$ consists of two blocks with exponent~$2$.

As for $H_2(A^\rho;\KK)$, let us consider its $(t-1)$-primary part whose exponent is given as $\rk\partial_2=1$. 
Using \autoref{thm:main}, the dimension of the $(t+1)$-primary part for $H_2(A^\rho;\KK)$ can be obtained by
\[
\n_{1,1}+2\n_{1,2}+3\n_{1,3}=
\sum_{j=0}^2 \tilde h_1(\cF^2_j)+\sum_{j=0}^1 \tilde h_1(\Gamma_j)-2\tilde{h}_1(\Gamma)= 3-2=1,
\]
i.e. $\n_{1,1}=1$, $\n_{1,2}=\n_{1,3}=0$, hence it is semisimple with one block:
$$
H_2(\bt^\rho;\KK)=\KK[t^{\pm 1}]\oplus
\frac{\KK[t^{\pm 1}]}{\langle t-1\rangle}\oplus
\frac{\KK[t^{\pm 1}]}{\langle t+1\rangle}.
$$
It is easy to check that $H_3(\bt^\rho;\KK)=0$.
\end{example}

\begin{example}
Let us consider the graph~$\Gamma$ of \autoref{fig:triangle} whose labels define a character~$\rho$. 
The flag complex~$\cF$ has dimension~$2$ (its realization is the big triangle in \autoref{fig:triangle}). 
Note that 
\[
\Gamma_0=\V,\quad \Gamma_1=\partial\cF,\quad \V_0=\{v_1,v_2,v_3\},\quad
\cF^2_0=\cF^2_1=\Gamma,\quad \cF_2^2=\overline{\cF\setminus\Delta}.
\]

\begin{figure}[ht]
\begin{center}
\begin{tikzpicture}
\Vertex[style={minimum size=1.0cm,shape=circle},LabelOut=false,L=\hbox{$1$},x=0.0cm,y=3cm]{v0}
\Vertex[style={minimum size=1.0cm,shape=circle},LabelOut=false,L=\hbox{$1$},x=-2.0cm,y=0.0cm]{v1}
\Vertex[style={minimum size=1.0cm,shape=circle},LabelOut=false,L=\hbox{$1$},x=2.0cm,y=0.0cm]{v2}
\Vertex[style={minimum size=1.0cm,shape=circle},LabelOut=false,L=\hbox{$2$},x=0.0cm,y=0cm]{v3}
\Vertex[style={minimum size=1.0cm,shape=circle},LabelOut=false,L=\hbox{$2$},x=-1.0cm,y=1.5cm]{v4}
\Vertex[style={minimum size=1.0cm,shape=circle},LabelOut=false,L=\hbox{$2$},x=1.0cm,y=1.5cm]{v5}
\Edge[lw=0.05cm](v0)(v4)
\Edge[lw=0.05cm](v4)(v1)
\Edge[lw=0.05cm](v1)(v3)
\Edge[lw=0.05cm](v3)(v2)
\Edge[lw=0.05cm](v2)(v5)
\Edge[lw=0.05cm](v5)(v0)
\Edge[lw=0.05cm](v4)(v3)
\Edge[lw=0.05cm](v3)(v5)
\Edge[lw=0.05cm](v5)(v4)
\node[right=8pt] at (v0) {$v_1$};
\node[left=8pt] at (v1) {$v_2$};
\node[right=8pt] at (v2) {$v_3$};
\node at ($1/3*(v3)+1/3*(v4)+1/3*(v5)$) {$\Delta$};
\end{tikzpicture}
\caption{}
\label{fig:triangle}
\vspace{-2mm}
\end{center}
\end{figure}

The dimension of the $(t-1)$-primary part for $H_1(A^\rho;\KK)$ is~$\rk\partial_1=5$ and it is semisimple.
The dimension of the $(t+1)$-primary part for $H_1(A^\rho;\KK)$ is $6+1-2-3=2$. 
The target of $\tilde{H}_0(\V_0)\to\tilde{H}_0(\Gamma_1)$ is zero; hence, the $(t+1)$-primary part 
for $H_1(A^\rho;\KK)$ consists of two blocks with exponent~$1$.

The $(t-1)$-primary part of $H_2(A^\rho;\KK)$ is $\rk\partial_2=4$ (and semisimple). The dimension
of the $(t+1)$-primary part for $H_2(A^\rho;\KK)$ is given by \autoref{thm:main} since 
\[
\tilde h_1(\cF_0^2)=\tilde h_1(\cF_1^2)=4,\quad \tilde h_1(\cF_2^2)=1,\quad \tilde h_1(\Gamma_0)=0,
\quad \tilde h_1(\Gamma_1)=1,\quad \tilde{h}_1(\Gamma)=4. 
\]
One has $\n_{1,1}+2\n_{1,2}+3\n_{1,3}=2$.
Since $H_1(\Gamma_0)=0$ there is no block with exponent~$3$ as expected. Moreover, the map
$H_1(\Gamma_1)\to H_1(\cF^2)$ is an isomorphism of vector spaces of dimension~$1$ and we deduce
that the $(t+1)$-primary part for $H_2(A^\rho;\KK)$ consists of $1$ block of exponent~$2$.
It is easily checked that $H_3(A^\rho;\KK)=0$.
\end{example}

\begin{example}
The graph $\Gamma$ together with the character $\rho$ defined by the labels as shown in \autoref{fig:block3}
provides an example of a maximal exponent $3$ block in $H_2(A^\rho;\KK)$. 

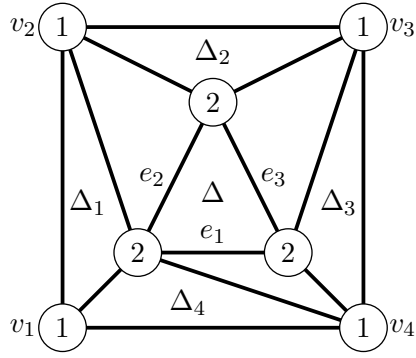
\begin{figure}[ht]
\begin{center}
\begin{tikzpicture}
\Vertex[style={minimum size=1.0cm,shape=circle},LabelOut=false,L=\hbox{$1$},x=-2cm,y=-2cm]{v0}
\Vertex[style={minimum size=1.0cm,shape=circle},LabelOut=false,L=\hbox{$1$},x=-2.0cm,y=2cm]{v1}
\Vertex[style={minimum size=1.0cm,shape=circle},LabelOut=false,L=\hbox{$1$},x=2.0cm,y=2.0cm]{v2}
\Vertex[style={minimum size=1.0cm,shape=circle},LabelOut=false,L=\hbox{$1$},x=2.0cm,y=-2cm]{v3}
\Vertex[style={minimum size=1.0cm,shape=circle},LabelOut=false,L=\hbox{$2$},x=-1.0cm,y=-1cm]{v4}
\Vertex[style={minimum size=1.0cm,shape=circle},LabelOut=false,L=\hbox{$2$},x=1.0cm,y=-1cm]{v5}
\Vertex[style={minimum size=1.0cm,shape=circle},LabelOut=false,L=\hbox{$2$},x=0.0cm,y=1cm]{v6}
\Edge[lw=0.05cm](v0)(v1)
\Edge[lw=0.05cm](v1)(v2)
\Edge[lw=0.05cm](v2)(v3)
\Edge[lw=0.05cm](v3)(v0)
\Edge[lw=0.05cm](v4)(v5)
\Edge[lw=0.05cm](v5)(v6)
\Edge[lw=0.05cm](v6)(v4)
\Edge[lw=0.05cm](v0)(v4)
\Edge[lw=0.05cm](v4)(v1)
\Edge[lw=0.05cm](v1)(v6)
\Edge[lw=0.05cm](v6)(v2)
\Edge[lw=0.05cm](v2)(v5)
\Edge[lw=0.05cm](v5)(v3) 
\Edge[lw=0.05cm](v3)(v4)
\node[left=6pt] at (v0) {$v_1$};
\node[left=6pt] at (v1) {$v_2$};
\node[right=6pt] at (v2) {$v_3$};
\node[right=6pt] at (v3) {$v_4$};
\node[above] at ($.5*(v4)+.5*(v5)$) {$e_1$};
\node[left] at ($.5*(v4)+.5*(v6)$) {$e_2$};
\node[right] at ($.5*(v5)+.5*(v6)$) {$e_3$};
\node at ($1/3*(v0)+1/3*(v1)+1/3*(v4)$) {$\Delta_1$};
\node at ($1/3*(v1)+1/3*(v2)+1/3*(v6)$) {$\Delta_2$};
\node at ($1/3*(v2)+1/3*(v3)+1/3*(v5)$) {$\Delta_3$};
\node at ($1/3*(v0)+1/3*(v3)+1/3*(v4)$) {$\Delta_4$};
\node at ($.3*(v4)+.3*(v5)+.4*(v6)$) {$\Delta$};
\end{tikzpicture}
\caption{The labeled graph $\Gamma$}
\label{fig:block3}
\end{center}
\end{figure}

Let us consider the filtered subcomplexes:
\begin{gather*}
\V_0=\{v_1,v_2,v_3,v_4\},\quad\Gamma_0=\V\cup\partial\cF,\quad 
\Gamma_1=\overline{\Gamma\setminus(e_1\cup e_2\cup e_3)},\quad 
\cF^2_0=\Gamma,\\
\cF^2_1=\Gamma\cup\Delta_1\cup\Delta_2\cup\Delta_3\cup\Delta_4,\quad 
\cF^2_2=\overline{\cF\setminus\Delta}.
\end{gather*}
Since $\Gamma_1$ is connected, $(t+1)$-primary part of $H_1(A^\rho;\KK)$ is semisimple.
Using \autoref{cor:H1}, its dimension is $4+1-2-3=0$, i.e., this primary part vanishes.
Hence $H_1(A^\rho;\KK)\cong\left(\KK[t^{\pm 1}]/\langle(t-1)\rangle\right)^6$.

The map $\tilde{H}_1(\Gamma_0)\to \tilde{H}_1(\cF^{2}_{2})$ is an isomorphism
between $1$-dimensional spaces; by \autoref{thm:main} 
$\n_{1,3}(2) = 1$, hence there is one copy $\KK[t^{\pm 1}]/\langle(t+1)^3\rangle$ 
in the $(t+1)$-primary part of $H_2(A^\rho;\KK)$.
Since 
\[
\n_{1,1}(2)+2\n_{1,2}(2)+3\n_{1,3}(2)\!=\!\sum_{j=0}^2\tilde{h}_1(\cF^2_j)+\sum_{j=0}^1\tilde{h}_1(\Gamma_j)-2\tilde{h}_1(\Gamma)=
(8+4+1)+(5+1)-2\cdot 8=3,
\]
one obtains
\[
H_2(A^\rho;\KK)\cong\left(\KK[t^{\pm 1}]/\langle(t-1)\rangle\right)^8\oplus\KK[t^{\pm 1}]/\langle(t+1)^3\rangle.
\]
Finally, one can check that $H_j^-(\mathbb{T}^{\rho_2};\KK)$ has dimension $0,0,1,1$ for $j=0,1,2,3$; as a consequence,
$H_3(A^\rho;\KK)=0$.
\end{example}


\begin{thebibliography}{1}

\bibitem{ChD:95}
R.~Charney and M.W. Davis, \emph{Finite {$K(\pi, 1)$}s for {A}rtin groups},
  Prospects in topology ({P}rinceton, {NJ}, 1994), Ann. of Math. Stud., vol.
  138, Princeton Univ. Press, Princeton, NJ, 1995, pp.~110--124.

\bibitem{binder}
Jupyter et~al., \emph{Binder 2.0 - {R}eproducible, interactive, sharable
  environments for science at scale}, Proc. 17th Python in Science Conf.
  (2018), {\tt 10.25080/Majora-4af1f417-011}.

\bibitem{sage90}
W.A.~Stein et~al., \emph{{S}age {M}athematics {S}oftware ({V}ersion 9.0)}, The
  Sage Development Team, 2019, {\tt http://www.sagemath.org}.

\bibitem{MevW:95}
J.~Meier and L.~VanWyk, \emph{The {B}ieri-{N}eumann-{S}trebel invariants for
  graph groups}, Proc. London Math. Soc. (3) \textbf{71} (1995), no.~2,
  263--280.

\bibitem{PS:09}
S.~Papadima and A.I. Suciu, \emph{Toric complexes and {A}rtin kernels}, Adv.
  Math. \textbf{220} (2009), no.~2, 441--477.

\bibitem{Sakuma-homology}
M.~Sakuma, \emph{Homology of abelian coverings of links and spatial graphs},
  Canad. J. Math. \textbf{47} (1995), no.~1, 201--224.

\bibitem{SV:13}
M.~Salvetti and A.~Villa, \emph{Combinatorial methods for the twisted
  cohomology of {A}rtin groups}, Math. Res. Lett. \textbf{20} (2013), no.~6,
  1157--1175.

\end{thebibliography}

\providecommand{\bysame}{\leavevmode\hbox to3em{\hrulefill}\thinspace}
\providecommand{\MR}{\relax\ifhmode\unskip\space\fi MR }
\providecommand{\MRhref}[2]{%
  \href{http://www.ams.org/mathscinet-getitem?mr=#1}{#2}
}
\providecommand{\href}[2]{#2}

\end{document}